\documentclass[10pt, reqno]{amsart}
\usepackage{graphicx, amssymb, amsmath, amsthm}
\numberwithin{equation}{section}

\usepackage{multicol}
\usepackage{color}

\usepackage{enumerate}
\usepackage{mathtools}
\usepackage[backref=page]{hyperref} 
\usepackage{amscd}
\def\pa{\partial}

\newcommand{\R}{\mathbb{R}}

\newcommand{\Z}{\mathbb{Z}}

\newtheorem{theorem}{Theorem}[section]

\newtheorem{lemma}[theorem]{Lemma}

\newtheorem{proposition}[theorem]{Proposition}

\theoremstyle{definition}
\newtheorem{definition}[theorem]{Definition}
\newtheorem{remark}[theorem]{Remark}

\makeatletter
\newcommand{\Extend}[5]{\ext@arrow0099{\arrowfill@#1#2#3}{#4}{#5}}
\makeatother
\let\pa=\partial

\let\d=\delta
\let\e=\varepsilon

\let\om=\omega

\let\Om=\Omega

\def\R{\mathbb{R}}
\def\Z{\mathbb{Z}}
\def\RRe{\mathrm{Re}}
\def\IIm{\mathrm{Im}}

\def\One{\uppercase\expandafter{\romannumeral 1}}
\def\Two{\uppercase\expandafter{\romannumeral 2}}
\def\Three{\uppercase\expandafter{\romannumeral 3}}
\def\Four{\uppercase\expandafter{\romannumeral 4}}
\begin{document}
\title[Defocusing NLS]{Almost sure scattering for defocusing energy critical Hartree equation on $\R^5$} 

\author{Liying Tao}

\address{Graduate School of China Academy of Engineering Physics, Beijing, China, 100088}

	\email{taoliying20@gscaep.ac.cn}

	\author[T. Zhao]{Tengfei Zhao}%
	\address{School of Mathematics and Physics, University of Science and Technology Beijing, Beijing
		100083, China}
	\email{zhao\underline{ }tengfei@ustb.edu.cn}%

\begin{abstract}
We consider the defocusing energy-critical Hartree equation $i\pa_tu+\Delta u=(|\cdot|^{-4}\ast|u|^2)u$ in spatial dimension $d=5$ and prove almost sure scattering with initial data $u_0\in H^s_x(\R^5)$ for any $s\in\R$. The proof relies on the modified interaction Morawetz estimate, the stability theories, the ``Narrowed'' Wiener randomization. We are inspired to consider this problem by the work of Shen-Soffer-Wu \cite{Shen-Soffer-Wu 1}, which treated the analogous problem for the energy-critical Schr\"{o}dinger equation. The new ingredient in this paper are that we take an alternative proof to give the interaction Morawetz estimate. And the nonlocal nonlinearity term will bring some difficulties.
\end{abstract}

 \maketitle

\begin{center}
 \begin{minipage}{100mm}
 \end{minipage}
 \end{center}

\section{Introduction}
In this paper, we consider the defocusing energy-critical Hartree equation:
\begin{equation}\label{NLS}
\begin{cases}
i\pa_tu+\Delta u=F(u),\ (t,x)\in\R\times\R^5,\\
u(0,x)=u_0(x),
\end{cases}
\end{equation}
where $u(t,x)$ is a complex-valued function in $\R\times\R^5$, $\Delta$ denotes the Laplacian in $\R^5$ and $F(u)=(V(\cdot)\ast|u|^2)u$ with $V(x)=|x|^{-4}$, the $\ast$ stands for the convolution in $\R^5$. The Hartree equation arises in the study of Boson stars and other physical phenomena, and in chemistry, it appears as a continuous-limit model for mesoscopic molecular structures, see for example \cite{Miao-Xu-Zhao 2011(PDE)} and references therein.


Define scaling transformation
\begin{equation}\label{scaling}
u_\lambda(t,x)\coloneqq \lambda^{\frac{3}{2}}u(\lambda^2t,\lambda x),\ \lambda>0.
\end{equation}
Clearly, it leaves equation \eqref{NLS} invariant, and the $\dot H^s(\R^5)$-norm of initial data behaves as 
\begin{equation}\label{H1 norm}
\|u_\lambda(0,x)\|_{\dot H_x^s(\R^5)}=\lambda^{s-1}\|u_0\|_{\dot H^s_x(\R^5)}.
\end{equation}
Note that, $\dot H^{1}_x(\R^5)$-norm ($s=1$ in \eqref{H1 norm}) is invariant under the scaling transformation, and we call the $\dot H^{1}_x(\R^5)$ as the critical Sobolev space. If the solution $u$ of \eqref{NLS} has sufficient smoothness and decay at infinity, it conserves mass and energy, which can be defined as follows,
\begin{equation}
M(u(t))\coloneqq\int_{\R^5}|u(t,x)|^2{\rm d}x=M(u_0),
\end{equation}
and
\begin{equation}
E(u(t))
\coloneqq
\frac{1}{2}\int_{\R^5}|\nabla u(t,x)|^2{\rm d}x+\frac{1}{4}\iint_{\R^5\times\R^5}\frac{|u(t,x)|^2|u(t,y)|^2}{|x-y|^4}{\rm d}x{\rm d}y
=E(u_0).
\end{equation}
Here, $E(u(t))$ is invariant under the scaling \eqref{scaling}. In view of this, the Cauchy problem \eqref{NLS} is called as energy-critical.

The Cauchy problem of the Hartree equation
\begin{equation}\label{NLS normal}
i\pa_t u+\Delta u=\mu(|x|^{-\gamma}\ast|u|^2)u,\ (t,x)\in\R\times\R^d
\end{equation}
has been intensively studied, where $\mu=\pm1$ and $0<\gamma<d$. 
\eqref{NLS normal} is referred as the defocusing case when $\mu=1$, and as the focusing case when $\mu=-1$. 
As is well known, Miao-Xu-Zhao \cite{Miao-Xu-Zhao 2008(PDE 22-44)} showed that the Cauchy problem \eqref{NLS normal} is locally well-posed in $H_x^s(\R^d)$ for $s\geq\max\{0,\frac{\gamma}{2}-1\}$ and small data scattering result in $H_x^1(\R^d)$ by using the Strichartz estimate and the contraction mapping argument. Ill-posedness for equation \eqref{NLS normal} in $H^s_x(\R^d)$ for any $s<\max\{0,\frac{\gamma}{2}-1\}$ has also been established in \cite{Miao-Xu-Zhao 2008(PDE 22-44)}. 

Many results have been obtained so far regarding the long-term dynamics of \eqref{NLS normal}. Using the method of Morawetz and Strauss \cite{Morawetz-Strauss 1972}, Ginibre-Velo \cite{Ginibre-Velo 1998(29-60)} developed the scattering theory of the defocusing $\dot H_x^1$-subcritical case $(2<\gamma<\min\{4,d\})$. Later, based on a new Morawetz estimate that is independent of the nonlinearity, Nakanishi gave an alternative proof for the result in \cite{Nakanishi 1999(107-118)}. Following the approach of Bourgain \cite{Bourgain 1998(267-297)}, Killip-Visan-Zhang \cite{Killip-Visan-Zhang 2009} and Tao \cite{Tao 2005(57-80)}, Miao-Xu-Zhao \cite{Miao-Xu-Zhao 2007(605-627)} obtained the global well-posedness and scattering results for the defocusing Hartree equation $i\pa_t u+\Delta u=(|x|^{-4}\ast|u|^2)u$ for the large radial data in $\dot H_x^1(\R^d)$, where $d\geq 5.$
Later, in \cite{Miao-Xu-Zhao 2011(PDE)}, the authors removed the radial assumption by using the frequency localized interaction Morawetz estimate. More detailed results of equation \eqref{NLS normal} can be found in \cite{{Cao-Guo},{Krieger-Lenzman-Raphael},{Miao-Xu-Zhao 2009(PDE 1831-1852)},{Miao-Xu-Zhao 2009(213-236)},{Miao-Xu-Zhao 2009(49-79)},{Miao-Xu-Zhao 2010(23-50)}}. 

From the seminal work in \cite{Christ-Colliander-Tao}, the Cauchy problem \eqref{NLS} with the initial data $H_x^s(\R^d)$ is ill-posed for any $s<1$. However, using the probabilistic tools, there has been a significant development in the probabilistic construction of local-in-time and global-in-time solutions to several nonlinear dispersive equations below the critical scaling regulartity. Building on the work of Lebowitz-Rose-Speer \cite{Lebowitz-Rose-peer}, Bourgain \cite{Bourgain-1994,Bourgain-1996} proved the invariance of the Gibbs measure under the flow of the equation and used this invariance to prove almost certain global well-posedness in the support of the measure. In \cite{Burq-Tzvetkov1,Burq-Tzvetkov2}, Burq and Tzvetkov studied the cubic nonlinear wave equation on a three-dimensional compact manifold. They constructed large sets of initial data of super-critical regularity, which lead to local solutions via a randomization procedure based on the expansion of the initial data with respect to an orthonormal basis of eigenfunctions of the Laplacian. Together with invariant measure considerations, they also proved almost sure global existence for the cubic nonlinear wave  equation (NLW) on the three-dimensional unit ball.

Next, we provide a concise summary of the global well-posedness and scattering result for the nonlinear Schr\"{o}dinger equation (NLS)
\begin{equation}\label{NLS1}
i\pa_tu+\Delta u=\mu|u|^pu,\ (t,x)\in\R\times\R^d, 
\end{equation}
and NLW
\begin{equation}\label{NLW}
\pa^2_tu-\Delta u+\mu|u|^p u=0,\ (t,x)\in\R\times\R^d
\end{equation}
with the random initial data. For NLS, by using the Wiener randomization, B$\acute{\text{e}}$nyi-Oh-Pocovnicu \cite{Benyi-Oh-Pocovnicu 2015} obtained that the cubic NLS is almost surely locally well-posed in $H_x^s(\R^d)$ with $d\geq 3$ and $s>\frac{(d-1)(d-2)}{2(d+1)}$. The authors also established that the cubic NLS is almost surely global well-posed and scattering for the small data in $H_x^s(\R^d)$, where $d\geq 3$ and $\frac{(d-1)(d-2)}{2(d+1)}<s\leq\frac{d-2}{2}$. Subsequently, the results were further improved in \cite{Benyi-Oh-Pocovnicu 2019} and \cite{Pocovnicu-Wang}. For NLW, based on Bourgain's high-low frequency decomposition and improved averaging effects for the free evolution of the randomized initial data, Luhrmann-Mendelson \cite{Luhrmann-Mendelson} proved almost sure existence of global solutions to NLW \eqref{NLW} with respect to the randomization of the initial data in $H_x^s(\R^3)\times H_x^{s-1}(\R^3)$ for $3\leq\rho<5$ and $\frac{\rho^3+5\rho^2-11\rho-3}{9\rho^2-6\rho-3}<s<1$. Later, the result of  \cite{Luhrmann-Mendelson} was improved in \cite{Luhrmann-Mendelson-2016,Sun-Xia}. Recently, Pocovnicu proved almost sure global existence and uniqueness for the energy-critical defocusing NLW \eqref{NLW} with the random initial data in $H_x^s(\R^d)\times H_x^{s-1}(\R^d)$, where $0<s\leq1$ for $d=4$ and $0\leq s\leq1$ for $d=5$. Combining the method in \cite{Pocovnicu1} with a uniform probabilistic energy bound for approximating the random solution, Oh-Pocovnicu \cite{Oh-Pocovnicu-1} proved almost sure global well-posedness of the energy-critical defocusing quintic NLW on $\R^3$ with the random initial data in $H_x^s(\R^d)\times H_x^{s-1}(\R^d)$ for $s>\frac{1}{2}$.

The first large date almost sure scattering result was proposed by Dodson-L\"{u}hrmann-Mendelson \cite{Dodson-Luhrmann-Mendelson 2020} for the energy-critical wave equation. 
In \cite{Dodson-Luhrmann-Mendelson 2020}, the authors considered the energy-critical defocusing NLW on $\R^4$ and obtained almost sure global existence and scattering theory for the randomized radially symmetric initial data in $H_x^s(\R^4)\times H_x^{s-1}(\R^d)$ for $\frac{1}{2}<s<1$ by using a radial Sobolev estimate and global-in-time random Strichartz estimate.
Then, the result in \cite{Dodson-Luhrmann-Mendelson 2020} was extended by Bringmann in \cite{{Bringmann 2020(1011-1050)},{Bringmann 2021(1931-1982)}}. Moreover, the methods of \cite{Dodson-Luhrmann-Mendelson 2020} were further developed by Killip-Murphy-Visan \cite{Killip-Murphy-visan}, and they proved that almost sure global existence and scattering for the energy-critical nonlinear Schr\"{o}dinger equation with randomized spherically symmetric initial data in $H^s(\R^4)$, where $\frac{5}{6}<s<1.$ The result in \cite{Killip-Murphy-visan} was then improved in \cite{Dodson-Luhrmann-Mendelson 2019} to $\frac{1}{2}<s<1$. For more probabilistic results of NLS \eqref{NLS1}, we refer to \cite{{Murphy 2019},{Nakanishi-Yamamoto},{Shen-Soffer-Wu 2},{Spitz},{Tao-Zhao}}. 

In \cite{Shen-Soffer-Wu 1}, Shen-Soffer-Wu considered the defocusing energy-critical NLS 
\begin{equation}\label{NLS2}
i\pa_tu+\Delta u=|u|^{\frac{4}{d-2}}u,\ (t,x)\in\R\times\R^d
\end{equation}
for $d=3$ and $d=4$. They proved almost sure scattering with non-radial data in $H_x^s(\R^d)$ for any $s\in\R$ by introducing a new version of randomization based on rescaled cube decomposition in frequency, called ``Narrowed" Wiener randomization. Luo \cite{Luo} developed a randomization that was adapted to the waveguide manifold $\R^3\times\mathbb{T}$ based on the ``Narrowed" Wiener randomization and proved the almost sure scattering for NLS \eqref{NLS2} in $H_x^s(\R^3\times\mathbb{T})$, where $s\in\R.$

In this paper, on the basis of  the ``Narrowed" Wiener randomization and random Strichartz estimate in \cite{Shen-Soffer-Wu 1}, we will focus on the almost sure global well-posedness and scattering theory for the defocusing energy-critical Hartree equation \eqref{NLS} for any $s\in\R$.

\subsection{Randomization}
Before stating the main result, we first recall the definition of the ``Narrowed" Wiener randomization in \cite{Shen-Soffer-Wu 1}. Let $a\in\mathbb{N}$ be a positive integer and $N\in2^{\mathbb{N}}$ be a dyadic number. Denote the cube sets
\begin{equation*}
O_N\coloneqq\{\xi\in\R^d:|\xi_j|\leq N, j=1,2,\cdots,d\},
\end{equation*}
and 
\begin{equation*}
Q_N\coloneqq O_{2N}\backslash O_N.
\end{equation*}
Define 
\begin{equation*}
\mathcal{A}(Q_N)
\coloneqq
\{Q:Q\ \text{is a dyadic cube with length}\ N^{-a}\text{ and }Q\subset Q_N\},
\end{equation*}
and 
$$\mathcal{Q}\coloneqq\{O_1\}\cup\{Q:Q\in\mathcal{A}(Q_N)\text{ and }N\in 2^{\mathbb{N}}\}.$$
By the above construction, $\R^d$ can be divided into
\begin{equation*}
\R^d=O_1\cup(\cup_{N\in2^{\mathbb{N}}}\cup_{Q\in\mathcal{A}(Q_N)}Q)\eqqcolon\cup_{Q\in\mathcal{Q}}Q.
\end{equation*}
Since $\mathcal{Q}$ is countable, the cubes in $\mathcal{Q}$ can be renumbered as follows
\begin{equation*}
\mathcal{Q}=\{Q_j:j\in\mathbb{N}\}.
\end{equation*}

The ``Narrowed" Wiener randomization is defined in the following.
\begin{definition}[{\cite[\emph{Definition 1.1}]{Shen-Soffer-Wu 1}}]\label{randomization}
Let $d\geq 1$ and $s\in\R$, the parameter $a\in\mathbb{N}$ satisfies
\begin{equation}\label{parameter a}
a>\max\left\{4-2s,5-4s,-\frac{1}{2}s,3\right\}.
\end{equation}
Let $\phi_j\in C_c^\infty(\R^d)$ be a real-valued bump function such that $\phi_j(\xi)=1$ for $\xi\in Q_j$ and $\phi_j(x)=0$ for $\xi\notin 2Q_j$, where $2Q_j$ is the cube with the same center as $Q_j$ and diam$(2Q_j)=$2diam$(Q_j)$. Let
\begin{equation*}
\psi_j(\xi)\coloneqq\frac{\phi_j(\xi)}{\sum_{k\in\mathbb{N}}\phi_k(\xi)}.
\end{equation*}
For $f:\R^d\to\mathbb{C}$, we define
\begin{equation}\label{Box j}
\Box_j f(x)=\mathcal{F}^{-1}_{\xi}(\psi_j(\xi)\mathcal{F}_xf(\xi))(x).
\end{equation}  

Let $\{g_j\}_{j\in\mathbb{N}}$ be a sequence of independent, mean-zero, complex Gaussian random variables on the probability space $(\Omega,\mathcal{A},\mathbb{P})$. Then, for any function $f$, we define the randomization of $f$ by
\begin{equation}\label{randomization of f}
f^\omega(x)\coloneqq\sum_{j\in\mathbb{N}}g_j(\omega)\Box_jf(x). 
\end{equation}
\end{definition}

\begin{remark}
The range of the parameter $a$ values should ensure that the bootstrap argument can be used to prove the almost conservation laws (Proposition \ref{almost conservation law}) and that the $Y$-norm can be bounded in the random sense.
\end{remark}

\subsection{Notion}
In the following, we denote that $A\lesssim B$ if $A\leq CB$ for some constants $C>0$, and $A\sim B$ means that $A\lesssim B$ and $B\lesssim A$. Moreover, we write $A\ll B$ if the implicit constant should be regarded as small. For any slab $I\subseteq\R$, the $X$-norm is defined as  
\begin{equation}
\|f\|_{X(I)}
\coloneqq
\|\langle\nabla\rangle f\|_{L_t^2(I;L_x^{\frac{10}{3}}(\R^5))}+\|f\|_{L_t^4(I;L_x^4(\R^5))}+\|f\|_{L_t^4(I;L_x^5(\R^5))}+\|f\|_{L_t^3(I;L_x^6(\R^5))},
\end{equation}
and the $Y$-norm is defined as
\begin{equation}\label{Y norm section 1}
\begin{split}
\|f\|_{Y(I)}
\coloneqq
&\|\langle\nabla\rangle^{s+\frac{1}{2}a}f\|_{L_t^2(I;L_x^5(\R^5))}+\|f\|_{L_t^4(I;L_x^5(\R^5))}+\|f\|_{L_t^4(I;L_x^4(\R^5))}\\
&+\|f\|_{L_t^6(I;L_x^3(\R^5))}.
\end{split}
\end{equation}

\subsection{Main result}
We are now in a position to state our results.
\begin{theorem}\label{main theorem}
Let $s\in\R$ and $u_0\in H_x^s(\R^5)$. Let $u_0^\omega$ be the associated randomized initial data as defined in \eqref{randomization of f}. Then, for almost every $\omega\in\Omega,$ there exists a global solution $u$ of \eqref{NLS} such that
\begin{equation}
u-e^{it\Delta}u_0^\omega\in C(\R;H_x^1(\R^5)).
\end{equation}
The solution $u$ scatters in the sense that there exists $u_\pm\in H_x^1(\R^5)$ such that
\begin{equation}\label{main th scattering}
\lim_{t\to\pm\infty}\|u-e^{it\Delta}u_0^\omega-e^{it\Delta}u_\pm\|_{H_x^1(\R^5)}=0.
\end{equation}
\end{theorem}

\begin{remark}
In the deterministic sense, scattering result cannot be obtained below the critical regularity for the Hartree equation \eqref{NLS}, see \cite{Christ-Colliander-Tao} and \cite{Miao-Xu-Zhao 2008(PDE 22-44)}. Nevertheless, we construct the global solution and obtain the scattering result to \eqref{NLS} below the energy space in the stochastic sense.

\end{remark}

\begin{remark}
The restriction for the dimension $d=5$ stems from the interaction Morawetz estimate (Lemma \ref{Classical interaction Morawetz}) and the random Strichartz estimate (Lemma \ref{random strichartz estimate}). For the details, we refer to Remark \ref{d>5 cannot}.
\end{remark}

Theorem \ref{main theorem} is in the spirit of the almost sure global well-posedness results in \cite{Shen-Soffer-Wu 1}. Namely, let $u_0\in H_x^s(\R^5)$ with $s\in\R$.
We expand the initial data $u_0(x)$ and the solution $u(t,x)$ to equation \eqref{NLS} as
\begin{equation}\label{v0 w0}
u_0(x)=P_{\geq N_0}u_0(x)+P_{<N_0}u_0(x)\eqqcolon v_0(x)+w_0(x),
\end{equation}
\begin{equation}\label{v w}
u(t,x)=e^{it\Delta}v_0+(u-e^{it\Delta}v_0)\eqqcolon v(t,x)+w(t,x), 
\end{equation}
where $N_0\gg 1$.
Then, \eqref{NLS} can be reformulated as the following perturbed equation
\begin{equation}\label{NLS-hartree-w}
\begin{cases}
i\pa_t w+\Delta w=(|\cdot|^{-4}\ast|w|^2)w+e,\\
w(0,x)=w_0(x),
\end{cases}
(t,x)\in\R\times\R^5,
\end{equation}
where $e=(|\cdot|^{-4}\ast|u|^2)u-(|\cdot|^{-4}\ast|w|^2)w.$
Define the energy $E_w(t)$ and mass $M_w(t)$ for the solution $w(t,x)$ of \eqref{NLS-hartree-w} as follows
\begin{equation}
E_w(t)\coloneqq\frac{1}{2}\int_{\R^5}|\nabla w(t,x)|^2{\rm d}x+\frac{1}{4}\iint_{\R^5\times\R^5}\frac{|u(t,x)|^2|u(t,y)|^2}{|x-y|^4}{\rm d}x{\rm d}y,
\end{equation}
\begin{equation}
M_w(t)\coloneqq\int_{\R^5}|w(t,x)|^2{\rm d}x.
\end{equation}

Next, we give the the global well-posedness and scattering of equation \eqref{NLS-hartree-w} by using perturbation theory. One of the key ingredients in the application of perturbation theory was the mass and energy bounds, see hypothesis \eqref{deterministic Condition2}. Moreover, the error $e=(|\cdot|^{-4}\ast|u|^2)u-(|\cdot|^{-4}\ast|w|^2)w$ can be made small for short time intervals thanks to the fact that the random linear part $v^\omega$ satisfies the random Strihartz estimate (Lemma \ref{random strichartz estimate}), which also ensures that the perturbation theory can be used.

\begin{theorem}\label{deterministic problem}
Let $a\in\mathbb{N}$ satisfy \eqref{parameter a}, $s\in\R$ and $A>0$. Then, there exists $N_0=N_0(A)\gg 1$ such that the following hold. Let $u_0\in H_x^s(\R^5)$, $v_0$ and $w_0$ satisfy \eqref{v0 w0}. Moreover, let $v$ and $w$ satisfy \eqref{v w} and \eqref{NLS-hartree-w}. Suppose that $v\in Y(\R)$, $w_0\in H_x^1(\R^5)$ such that
\begin{equation}\label{deterministic Condition1}
\|u_0\|_{H_x^s(\R^5)}+\|v\|_{Y(\R)}\leq A,
\end{equation}
and
\begin{equation}\label{deterministic Condition2}
M_w(0)\leq AN_0^{-2s},\  E_w(0)\leq AN_0^{2(1-s)}.
\end{equation}
Then, there exists a solution $w$ to \eqref{NLS-hartree-w} in $ C(\R;H_x^1(\R^5))$.

Furthermore, there exists a solution $u=v+w$ to \eqref{NLS} on $\R$ and  $u_{\pm}\in H_x^1(\R^5)$ such that
\begin{equation*}
\lim_{t\to\pm\infty}\|u(t)-v(t)-e^{it\Delta}u_\pm\|_{H_x^1(\R^5)}=0.
\end{equation*}
\end{theorem}

\begin{remark}\label{remark deterministic}
Theorem \ref{deterministic problem} shows that if $v\in Y(\R)$, $w_0\in H_x^1(\R^5)$, condition \eqref{deterministic Condition1} and the condition \eqref{deterministic Condition2} are satisfied, the global well-posedness of equation \eqref{NLS-hartree-w} can be obtained. Furthermore, the global well-posedness and the scattering of Hartree equation \eqref{NLS} are derived. Using the random Strichartz estimate for the random linear part $v^\omega=e^{it\Delta}P_{\geq N_0}u^\omega_0(x)$, $w_0^\omega=P_{< N_0}u^\omega_0(x)$ (the low frequency part of the initial value $u^\omega_0$) and $w$, we can get 
\begin{equation}\label{random sense result1}
v^\omega\in Y(\R),\  w^\omega_0\in H_x^1(\R^5),   
\end{equation}
and there exists $A>0$ such that
\begin{equation}\label{random sense result2}
\|u^\omega_0\|_{H_x^s(\R^5)}+\|v^\omega\|_{Y(\R)}\leq A,\  M_{w^\omega}(0)\leq AN_0^{-2s},\  E_{w^\omega}(0)\leq AN_0^{2(1-s)}.
\end{equation}
Here,
\begin{equation*}
E_{w^\omega}(t)\coloneqq\frac{1}{2}\int_{\R^5}|\nabla w^\omega(t,x)|^2{\rm d}x+\frac{1}{4}\iint_{\R^5\times\R^5}\frac{|u^\omega(t,x)|^2|u^\omega(t,y)|^2}{|x-y|^4}{\rm d}x{\rm d}y,
\end{equation*}
\begin{equation*}
M_{w^\omega}(t)\coloneqq\int_{\R^5}|w^\omega(t,x)|^2{\rm d}x.
\end{equation*}
Combining the random sense results \eqref{random sense result1} and \eqref{random sense result2} with Theorem \ref{deterministic problem}, Theorem \ref{main theorem} can be proved, see Section \ref{section 6}.
\end{remark}

\begin{remark}\label{s<0}
The proof of the related result in Theorem \ref{deterministic problem} is straight when $0<s<1$, since the mass conservation law is available. Hence, we only need to show Theorem \ref{deterministic problem} for $s<0$.
\end{remark}

\begin{remark}
The proof process of Theorem
\ref{deterministic problem} is briefly summarised as follows:

(1). First, we introduce the random Strichartz estimate (Lemma \ref{random strichartz estimate}), which was established by Shen-Soffer-Wu \cite{Shen-Soffer-Wu 1}. Then, unlike the proof in \cite{Shen-Soffer-Wu 1}, we take an alternative proof to give the interaction Morawetz estimate for the solution $w$ of \eqref{NLS-hartree-w} in Proposition \ref{morawetz NLS-hartree-w}.

(2). Using the fixed-point theorem, we prove the local well-posedness for the perturbed equation \eqref{NLS-hartree-w}. Specifically, we obtain that there exists a unique solution $w$ of \eqref{NLS} in
\begin{equation*}
B_{\delta,T}\coloneqq\{w\in C([0,T];H_x^1):\|w\|_{L_t^\infty H_x^1}\leq2CR,\ \|w\|_{X([0,T])}\leq 2C\delta\}, 
\end{equation*}
where $\delta$ is sufficient small. 
We also give the perturbation lemmas (Lemma \ref{short stability}, Lemma \ref{Long stability}) and the global-time bound for the Hartree equation \eqref{tilde w} (Lemma \ref{global space-time bound}), which were the main method to prove Theorem \eqref{deterministic problem}.

(3). A fundamental difficulty in the study of the perturbed equation \eqref{NLS-hartree-w} is that it is no longer a Hamiltonian equation, which lead to no conserved functional, and to obtain energy and mass bounds is non-trivial. However, by using the Morawetz estimate \eqref{interaction morawetz estimate} and the bootstrap argument, we obtain the almost conservation laws for the solution $w$ under the assumption that the mass $M_w(0)$ and the energy $E_w(0)$ are bounded.

(4). Finally, we split $\R$ into small time intervals $I_k$ such that $\R=\cup I_k$ and $\|v\|_{Y(I_k)}$ is sufficiently small. Based on the almost conservation laws deduced from Lemma \ref{almost conservation law},  we can apply the perturbation lemma (Lemma \ref{Long stability}) iteratively to any small time interval $I_k$. By summing the estimates on each interval, Theorem \eqref{deterministic problem} can be proved.

Once the proof of Theorem \ref{deterministic problem} is completed, Theorem \ref{main theorem} can also be proved, see Remark \ref{remark deterministic}.
\end{remark}

We next give the outline of this paper. In Section \ref{section 2}, we present some deterministic and probabilistic lemmas. In Section \ref{section 3}, we give the interaction Morawetz estimate for the solution $w$ of equation \eqref{NLS-hartree-w}. In Section \ref{section 4}, we prove the local well-posedness and establish the stability theories. In Section \ref{section 5}, we show the almost conservation laws and give a proof of Theorem \ref{deterministic problem}. Finally, we give a detailed proof of Theorem \ref{main theorem} in Section \ref{section 6}. 

\section{Basic tools and some elementary Properties}\label{section 2}
In this section, we recall some tools and useful results for the Hartree equation \eqref{NLS}. 

\subsection{Deterministic results}
The Fourier transform on $\R^d$ is defined as
\begin{equation*}
\hat f(\xi)\coloneqq(2\pi)^{-d/2}\int_{\R^d}f(x)e^{-ix\cdot\xi}{\rm d}x,
\end{equation*}
which leads to the fractional differentiation operators $|\nabla|^s$ and $\langle\nabla\rangle^s$ defined by
\begin{equation*}
\widehat{|\nabla|^sf}(\xi)\coloneqq|\xi|^s\hat f(\xi),\quad\widehat{\langle\nabla\rangle^sf}(\xi)\coloneqq\langle\xi\rangle^s\hat f(\xi),
\end{equation*}
where $\langle\xi\rangle\coloneqq \sqrt{1+|\xi|^2}.$ The homogeneous and inhomogeneous Sobolev norms can be defined as 
\begin{align*}
\|f\|_{\dot H_x^s(\R^d)}&\coloneqq\||\xi|^s\hat f\|_{L_x^2(\R^d)},\\
\|f\|_{H_x^s(\R^d)}&\coloneqq\|\langle\xi\rangle^s\hat f\|_{L_x^2(\R^d)}.\\
\end{align*}
Let $\varphi\in C_c^\infty(\R^d)$ supported on the ball $|\xi|\leq 2$ and equal to $1$ on the ball $|\xi|\leq 1$. For $N=2^k$, $k\in\Z$, we define the Littlewood-Paley projection operators by
\begin{align*}
\widehat{P_{\leq N}f}(\xi)&\coloneqq\varphi(\xi/N)\hat f(\xi),\\
\widehat{P_{>N}f}(\xi)&\coloneqq(1-\varphi(\xi/N))\hat f(\xi),\\
\widehat{P_{N}f}(\xi)&\coloneqq(\varphi(\xi/N)-\varphi(2\xi/N))\hat f(\xi).
\end{align*}  
These operators also obey the following standard Bernstein estimates.
\begin{lemma}[Bernstein estimates]\label{Bernstein estimates}
For $1\leq p\leq q\leq\infty$ and $s\geq 0$,
\begin{align*}
\||\nabla|^{\pm s}P_Nf\|_{L_x^q(\R^d)}&\sim N^{\pm s}\|P_Nf\|_{L_x^q(\R^d)},\\
\||\nabla|^{s}P_{\leq N}f\|_{L_x^q(\R^d)}&\lesssim N^{s}\|P_{\leq N}f\|_{L_x^q(\R^d)},\\
\|P_{>N}f\|_{L_x^q(\R^d)}&\lesssim N^{-s}\||\nabla|^sP_{>N}f\|_{L_x^q(\R^d)},\\
\|P_{\leq N}f\|_{L_x^q(\R^d)}&\lesssim N^{\left(\frac{d}{p}-\frac{d}{q}\right)}\|P_{\leq N}f\|_{L_x^p(\R^d)}.
\end{align*}
\end{lemma}

In the following, we present some properties of $\Box_j$, which is defined as in \eqref{Box j}.
\begin{lemma}[Orthogonality {\cite[\emph{Lemma 2.6}]{Shen-Soffer-Wu 1}}]
Let $f\in L_x^2(\R^d)$. Then, we have
\begin{equation*}
\|\Box_jf\|_{L_x^2(\R^d;\ell_{j}^2(\mathbb{N}))}\sim\|f\|_{L_x^2(\R^d)}.
\end{equation*}
\end{lemma}

\begin{lemma}[$L^q$-$L^p$ estimate {\cite[\emph{Lemma 2.7}]{Shen-Soffer-Wu 1}}]\label{p-q estimate}
Let $a>0$ and $2\leq p\leq q\leq\infty$. Given any $j\in\mathbb{N}$, then 
\begin{equation*}
\|\Box_j f\|_{L_x^q(\R^d)}\lesssim\|\langle\nabla\rangle^{-a(\frac{d}{p}-\frac{d}{q})}\Box_j f\|_{L_x^p(\R^d)}.
\end{equation*}
\end{lemma}

Next, we review some well-known lemmas that will be used later in the proof of Proposition \ref{morawetz NLS-hartree-w}.
\begin{lemma}[{\cite[\emph{Theorem 5.9}]{Lieb-Loss}}]\label{Lieb-Loss estimate}
For $d\geq 4$, $t\in I\subset\R$, $x\in\R^d$ and $y\in\R^d$, we have
\begin{equation*}
\big\||\nabla|^{-\frac{d-3}{2}}(|u|^2)\big\|^2_{L_t^2(I;L_x^2(\R^d))}\simeq\int_I\iint_{\R^d\times\R^d}\frac{|u(t,x)|^2|u(t,y)|^2}{|x-y|^3}{\rm d}x{\rm d}y{\rm d}t. 
\end{equation*}
In particular, for $d=5$, 
\begin{equation}
\big\||\nabla|^{-1}(|u|^2)\big\|^2_{L_t^2(I;L_x^2(\R^5))}\simeq\int_I\iint_{\R^5\times\R^5}\frac{|u(t,x)|^2|u(t,y)|^2}{|x-y|^3}{\rm d}x{\rm d}y{\rm d}t. 
\end{equation}
\end{lemma}

\begin{lemma}[{\cite[\emph{Lemma 5.6}]{Visan}}]\label{Visan estimate}
Let $d\geq 1$, we have
\begin{equation*}
\big\||\nabla|^{-\frac{d-3}{4}}f\big\|_{L_x^4(\R^d)}^2\lesssim\||\nabla|^{-\frac{d-3}{2}}|f|^2\|_{L_x^2(\R^d)}.
\end{equation*}
In particular, for $d=5$, 
\begin{equation}
\big\||\nabla|^{-\frac{1}{2}}f\big\|_{L_x^4(\R^5)}^2\lesssim\||\nabla|^{-1}|f|^2\|_{L_x^2(\R^5)}.
\end{equation}  
\end{lemma}

\begin{lemma}[Gagliardo-Nirenberg inequality {\cite[\emph{Lemma 2.3}]{Cardoso-Guzman-Pastor 2022}}]
Let $d\geq 1$, $0<\theta<1$, $0<s_1<s_2$ and $1<p_1,p_2,p_3\leq\infty.$ Suppose that $\frac{1}{p_1}=\frac{\theta}{p_2}+\frac{1-\theta}{p_3}$ and $s_1=\theta s_2$. Then, we have
\begin{equation}\label{G-N}
\||\nabla|^{s_1}u\|_{L_x^{p_1}(\R^d)}
\lesssim
\||\nabla|^{s_2}u\|_{L_x^{p_2}(\R^d)}^\theta\|u\|_{L_x^{p_3}(\R^d)}^{1-\theta}.
\end{equation}
\end{lemma}

\begin{lemma}[Hardy's inequality {\cite[\emph{Theorem 9.5}]{Muscalu-Schlag}}]
For $0<s<\frac{d}{2}$, we have that
\begin{equation}\label{Hardy}
\||x|^{-s}u\|_{L_x^2(\R^d)}\lesssim\|u\|_{\dot H_x^s(\R^d)}.
\end{equation}
\end{lemma}

\begin{lemma}[Hardy-Littlewood-Sobolev inequality {\cite[\emph{Theorem 1.7}]{Bahouri-Chemin-Danchin}}]\label{Hardy-Littlewood Sobolev inequality}
Let $1<p<q<\infty$, $f\in L_x^p(\R^5)$. Then under the assumption that
\begin{equation*}
\frac{1}{q}=\frac{1}{p}-\frac{1}{5},
\end{equation*}
there exists a positive constant $C$ such that
\begin{equation}\label{Hardy-Littlewood Sobolev inequality formula}
\big\||\cdot|^{-4}\ast f\big\|_{L_x^q(\R^5)}\leq C\|f\|_{L_x^p(\R^5)}.
\end{equation}
\end{lemma}

For later use, we recall the classical Strichartz estimate.
\begin{definition}
A pair of Lebesgue space exponents $(q,r)$ are called $\alpha$-admissible for $\R^{1+d}$, or denote by $(q,r)\in\Lambda_\alpha$ if 
\begin{equation*}
d\left(\frac{1}{2}-\frac{1}{r}\right)-\frac{2}{q}=\alpha.
\end{equation*}
In particular, $(q,r)$ is a Schr\"{o}dinger admissible pair when $\alpha=0.$ If $(q,r)\in\Lambda_\alpha$ and $\alpha\geq 0$, we define that
\begin{equation*}
(q,r)\in\cup_{\alpha\geq 0}\Lambda_{\alpha}.
\end{equation*}
\end{definition}

\begin{lemma}[\cite{{Cazenave},{Keel-Tao}} ]\label{classical Strichatz}
Let $(q,r)$ and $(\tilde q,\tilde r)$ be two arbitrary Schr\"{o}dinger admissible pairs, $I\subset\R$ be a time interval and $0\in I$. Suppose $u$ is a solution to
\begin{equation*}
\begin{cases}
i\pa_t u+\Delta u=G(t,x),\ (t,x)\in I\times\R^d,\\
u(0,x)=f(x)\in L_x^2(\R^d).
\end{cases}   
\end{equation*}
Then, there exists $C_0>0$ such that
\begin{equation}
\|u\|_{L_t^q(I;L_x^r(\R^d))}
\leq
C_0\left(\|f\|_{L_x^2(\R^d)}+\|G\|_{L_t^{\tilde q'}(I;L_x^{\tilde r'}(\R^d))}\right),
\end{equation}
where the primed exponents denote H\"{o}lder dual exponents.
\end{lemma}

\subsection{Probabilistic results}
In addition to the above deterministic estimates, probabilistic estimates also need to be introduced. Firstly, we recall the large deviation estimate.

\begin{lemma}[Large deviation estimate {\cite[\emph{Lemma 3.1}]{Burq-Tzvetkov1}}]\label{L-D estimate1}
Let $\{h_j(\om)\}_{j=1}^\infty$ be a sequence of real, $0$-mean, independent random variables with associated sequence of distributions $(\mu_j)_{j=1}^\infty$ on a probability space $(\Omega,\mathcal{A},\mathbb{P})$. Assume that there exists $c>0$ such that
\begin{align*}
\forall\gamma\in\R,\forall j\geq 1,
\left|\int_{-\infty}^\infty e^{\gamma x}{\rm d}\mu_j(x)\right|
\leq e^{c\gamma^2}.
\end{align*}
Then there exists $\beta>0$ such that for every $\lambda>0$, every sequence $\{c_j\}_{j=1}^\infty\in\ell^2$ of real numbers,
\begin{align*}
\mathbb{P}\left(\om:\left|\sum_{j=1}^\infty c_j h_j(\om)\right|>\lambda\right)
\leq 2e^{-\frac{\beta\lambda^2}{\sum_n c_j^2}}.
\end{align*}
As a consequence for every $p\geq 2$, $\{c_j\}_{j=1}^\infty\in\ell^2$, we have
\begin{align*}
\left\|\sum_{j=1}^\infty c_jh_j(\om)\right\|_{L_\omega^p(\Om)}
\lesssim
\sqrt{p}\left(\sum_{j=1}^\infty c_j^2\right)^{1/2}.
\end{align*}
\end{lemma}
\begin{remark}
Note that, the distributions $\theta_j$ of the Gaussian random variables $g_j(\omega)$
\begin{equation*}
{\rm d}\theta_j(x)=(2\pi)^{-\frac{1}{2}}e^{-\frac{x^2}{2}}{\rm d}x
\end{equation*}
satisfy:
\begin{equation*}
\exists c>0,\ \left|\int_{\R}e^{\gamma x}{\rm d}\theta_j\right|\leq e^{c\gamma^2}
\end{equation*}
for all $\gamma\in\R$ and $j\in\mathbb{N}$.  
\end{remark}


\begin{lemma}[{\cite[\emph{Lemma 2.5}]{Luhrmann-Mendelson}}]\label{L-D estimate2}
Suppose $F$ is a measurable function on a probability space $(\Omega,\mathcal{A},\mathbb{P})$, and suppose that there exist $C>0,A>0$ and $p_0\geq 1$ such that for $p\geq p_0$, we have 
\begin{equation*}
\|F\|_{L^{p}_{\om}(\Om)}\leq C\sqrt{p}A.
\end{equation*}
Then there exists $C'=C'(C,p_0)>0$ and $c=c(C,p_0)>0$ such that for every $\lambda>0$,
\begin{equation*}
\mathbb{P}(\om\in\Om:|F(\om)|>\lambda)\leq C'{\rm exp}\{-c\lambda^2A^{-2}\}.
\end{equation*}
\end{lemma}

In the following, we review the random Strichartz estimate that appears in \cite{Shen-Soffer-Wu 1}. 

\begin{lemma}[Random Strichartz estimate {\cite[\emph{Lemma 3.1}]{Shen-Soffer-Wu 1}}]\label{random strichartz estimate}
Let $d\geq 1$ and $f\in L_x^2(\R^d)$. Define the randomization $f^\omega$ as in Definition \ref{randomization}. Suppose $2\leq q,r<\infty$ with $(q,r)\in \cup_{\alpha\geq 0}\Lambda_{\alpha}$, and let $2\leq r_0\leq r$ such that $(q,r_0)\in\Lambda_0$. Then for any $p>\max\{q,r\}$ and $0\leq s\leq a\left(\frac{d}{r_0}-\frac{d}{r}\right)$, 
\begin{equation}\label{random estimate 1}
\|\langle\nabla\rangle^s e^{it\Delta}f^\omega\|_{L_\omega^p(\Omega;L_t^q(\R;L_x^r(\R^d)))}\lesssim\sqrt p\|f\|_{L_x^2(\R^d)}.
\end{equation}
\end{lemma}

\section{Interaction morawetz estimate}\label{section 3}
This section aims to establish the interaction Morawetz of the solution $w$ for equation \eqref{NLS-hartree-w}. We adopt the convention that repeated indices are summed throughout this section. For $f,g$, define the momentum brackets by
\begin{equation*}
\{f,g\}_p=\RRe(f\nabla\bar g-g\nabla\bar f).
\end{equation*}

Recall the standard Morawetz action centered at a point. Let $m$ be a function on the slab $I\times\R^d$, and suppose $\phi$ satisfies
\begin{equation*}
i\pa_t\phi+\Delta\phi=F
\end{equation*}
on $I\times\R^d$. We define the Morawetz action centered at zero to be
\begin{equation*}
M_m^0(t)=2\int_{\R^d}m_j(x)\IIm(\bar\phi(x)\phi_j(x)){\rm d}x.
\end{equation*}
Here, $m_j(x)=\pa_{x_j}m(x)$.

To begin with, we introduce the following lemma. It can be found in \cite[\emph{Lemma 6.1}]{Miao-Xu-Zhao 2011(PDE)} and \cite[\emph{Lemma 5.3}]{Visan}.
\begin{lemma}\label{Classical Morawetz}
Let $m$ be a smooth weight, then
\begin{equation}
\pa_t M_m^0
=\int_{\R^d}(-\Delta\Delta m)|\phi|^2
+4\int_{\R^d}m_{jk}\RRe(\bar\phi_j\phi_k)
+2\int_{\R^d}m_j\{F,\phi\}_p^j,
\end{equation}
where $\{F,\phi\}_p^j=\RRe(F\pa_j\bar\phi-\phi\pa_j\bar F).$
\end{lemma}

Next, we define the interaction Morawetz action
\begin{equation}
\begin{split}
M(t)
&\coloneqq
\int_{\R^d}|\phi(y)|^2M_m^y(t){\rm d}t\\
&=2\IIm\iint_{\R^d\times\R^d}|\phi(y)|^2\nabla m(x-y)\cdot\nabla \phi(x)\bar\phi(x){\rm d}x{\rm d}y.
\end{split}
\end{equation}
Note that
\begin{equation}\label{Classical Morawetz1}
\pa_t(|\phi|^2)
=2\IIm(-\Delta\phi\bar\phi+F\bar\phi)
=-2\pa_j\IIm(\pa_j\phi\bar\phi)+2\IIm(F\bar\phi),
\end{equation}
and the subsequent lemma follows directly from Lemma \ref{Classical Morawetz} and \eqref{Classical Morawetz1}.

\begin{lemma}\label{Classical interaction Morawetz}
\begin{subequations}
\begin{align}
\pa_tM(t)
=&\int_{\R^d}\pa_t(|\phi(y)|^2)M_m^y(t){\rm d}t
+\int_{\R^d}|\phi(y)|^2\pa_tM_m^y(t){\rm d}t\notag\\
=&-4\iint_{\R^d\times\R^d}\IIm(\pa_j\phi\bar\phi)(y)m_{jk}(x-y)\IIm(\pa_k\phi\bar\phi)(x){\rm d}x{\rm d}y\label{C-Ma}\\
&+4\iint_{\R^d\times\R^d}\IIm(F\bar\phi)(y)\nabla m(x-y)\cdot\IIm(\nabla\phi\bar\phi)(x){\rm d}x{\rm d}y\label{C-Mb}\\
&+\iint_{\R^d\times\R^d}(-\Delta\Delta m)(x-y)|\phi(x)|^2|\phi(y)|^2{\rm d}x{\rm d}y\label{C-Mc}\\
&+4\iint_{\R^d\times\R^d}|\phi(y)|^2m_{jk}(x-y)\RRe(\pa_j\bar\phi\pa_k\phi)(x){\rm d}x{\rm d}y\label{C-Md}\\
&+2\iint_{\R^d\times\R^d}|\phi(y)|^2m_j(x-y)\{F,\phi\}_p^j(x){\rm d}x{\rm d}y\label{C-Me}.
\end{align}
\end{subequations}
\end{lemma}

Using Lemma \ref{Classical interaction Morawetz}, we derive the interaction Morawetz estimate of the solution $w$ for equation \eqref{NLS-hartree-w}.

\begin{proposition}\label{morawetz NLS-hartree-w}
Let $w$ be the solution of \eqref{NLS-hartree-w}. Then for $[0,T]\subset\R$ we have
\begin{equation}\label{interaction morawetz estimate}
\begin{split}
\big\||\nabla|^{-\frac{1}{2}}w\big\|_{L_{t,x}^4}^4
\lesssim
&\|w(t)\|_{L_t^\infty L_x^2}^2\|w(t)\|_{L_t^\infty\dot H_x^{\frac{1}{2}}}^2\\
&+\|v\|_{L_t^2L_x^5}
\left(\|w\|_{L_t^\infty H^1_x}\big\||\nabla|^{-\frac{1}{2}}w\right\|_{L_{t,x}^4}^2+\|v\|_{L_t^6L_x^3}^3\big)\|w\|^2_{L_t^\infty\dot H^{\frac{1}{2}}_y}\\
&+\|\nabla v\|_{L_t^2L_x^5}
\left(\|w\|_{L_t^\infty H^1_x}\big\||\nabla|^{-\frac{1}{2}}w\right\|_{L_{t,x}^4}^2+\|v\|_{L_t^6L_x^3}^3\big)\|w\|_{L_t^\infty L_y^2}^2,
\end{split}
\end{equation}
where all the space-time norms are taken over $[0,T]\times\R^5.$
\end{proposition}
\begin{proof}
Let $\phi=w$ and $F=F(w)=(|\cdot|^{-4}\ast |w|^2)w+e$ in Lemma \ref{Classical interaction Morawetz},  where $e=(|\cdot|^{-4}\ast|u|^2)u-(|\cdot|^{-4}\ast |w|^2)w$, we derive that:
\begin{equation*}
\begin{split}
\eqref{C-Mb}
=&
4\iint\IIm((|\cdot|^{-4}\ast |w|^2)|w|^2+e\bar w)(y)\nabla m(x-y)\cdot\IIm(\nabla w\bar w)(x){\rm d}x{\rm d}y\\
=&4\iint\IIm((|\cdot|^{-4}\ast |w|^2)|w|^2)(y)\nabla m(x-y)\cdot\IIm(\nabla w\bar w)(x){\rm d}x{\rm d}y\\
&+4\iint\IIm(e\bar w)(y)\nabla m(x-y)\cdot\IIm(\nabla w\bar w)(x){\rm d}x{\rm d}y\\
=&4\iint\IIm(e\bar w)(y)\nabla m(x-y)\cdot\IIm(\nabla w\bar w)(x){\rm d}x{\rm d}y,
\end{split}
\end{equation*}
and 
\begin{equation*}
\begin{split}
\eqref{C-Me}
=&
2\iint|w(y)|^2m_j(x-y)\{(|\cdot|^{-4}\ast |w|^2)w+e,w\}_p^j(x){\rm d}x{\rm d}y.
\end{split}
\end{equation*}
Note that
\begin{equation*}
\begin{split}
&\{(|\cdot|^{-4}\ast |w|^2)w+e,w\}_p^j\\
=&
\RRe\left[\left((|\cdot|^{-4}\ast |w|^2)w+e\right)\pa_j\bar w-w\pa_j\left(\overline{(|\cdot|^{-4}\ast |w|^2)w+e}\right)\right]\\
=&
-2\RRe\left[(|\cdot|^{-4}\ast\pa_j w\bar w)|w|^2\right]
+\RRe(e\pa_j\bar w-\pa_j\bar e w).
\end{split}
\end{equation*}
Hence, we get
\begin{equation*}
\begin{split}
\eqref{C-Me}
=&
-4\RRe\iint|w(y)|^2m_j(x-y)\big[(|\cdot|^{-4}\ast\pa_j w\bar w)|w|^2\big](x){\rm d}x{\rm d}y\\
&+2\RRe\iint|w(y)|^2m_j(x-y)(e\pa_j\bar w-\pa_j\bar e w)(x){\rm d}x{\rm d}y\\
=&
-\RRe\iint|w(y)|^2m_j(x-y)\pa_j\big[(|\cdot|^{-4}\ast |w|^2)|w|^2\big](x){\rm d}x{\rm d}y\\
&+2\RRe\iint|w(y)|^2m_j(x-y)(2e\pa_j\bar w-\pa_j(e\bar  w))(x){\rm d}x{\rm d}y.
\end{split}
\end{equation*}
According to the above,
\begin{subequations}\label{Classical Morawetz 1}
\begin{align}
\pa_tM(t)
=&\int_{\R^d}\pa_t(|w(y)|^2)M_m^y(t){\rm d}t
+\int_{\R^d}|w(y)|^2\pa_tM_m^y(t){\rm d}t\notag\\
=&-4\iint\IIm(\pa_jw\bar w)(y)m_{jk}(x-y)\IIm(\pa_kw\bar w)(x){\rm d}x{\rm d}y\label{C-Ma1}\\
&+\iint(-\Delta\Delta m)(x-y)|w(x)|^2|w(y)|^2{\rm d}x{\rm d}y\label{C-Mb1}\\
&+4\iint|w(y)|^2m_{jk}(x-y)\RRe(\pa_j\bar w\pa_k w)(x){\rm d}x{\rm d}y\label{C-Mc1}\\
&+\RRe\iint|w(y)|^2\Delta m(x-y)\big[(|\cdot|^{-4}\ast |w|^2)|w|^2\big](x){\rm d}x{\rm d}y\label{C-Md1}\\
&+4\iint\IIm(e\bar w)(y)\nabla m(x-y)\cdot\IIm(\nabla w\bar w)(x){\rm d}x{\rm d}y\label{C-Me1}\\
&+4\RRe\iint|w(y)|^2\nabla m(x-y)(e\nabla\bar w)(x){\rm d}x{\rm d}y\label{C-Mf1}\\
&+2\RRe\iint|w(y)|^2\Delta m(x-y)(e\bar  w)(x){\rm d}x{\rm d}y\label{C-Mg1}.
\end{align}
\end{subequations}
Let $m(x-y)=|x-y|.$ Then for $d=5,$ an easy calculation shows that
\begin{equation*}
\begin{split}
\nabla m(x-y)&=\frac{x-y}{|x-y|},\\
m_{jk}(x-y)&=\frac{\d_{jk}}{|x-y|}-\frac{(x_j-y_j)(x_k-y_k)}{|x-y|^3},\\
\Delta m(x-y)&=\frac{4}{|x-y|},
\ \Delta\Delta m(x-y)=-\frac{8}{|x-y|^3}.
\quad
\end{split}
\end{equation*}
In this case,
\begin{equation}
\begin{split}
|M(t)|
&=\left|2\IIm\iint\frac{x-y}{|x-y|}|w(y)|^2\cdot(\nabla w\bar w)(x){\rm d}x{\rm d}y\right|\\
&\lesssim\|w(t)\|_{L_x^2}^2\|w(t)\|_{\dot H_x^{\frac{1}{2}}}^2.
\end{split}
\end{equation}

We claim that
\begin{equation}\label{01}
{\eqref{C-Ma1}+\eqref{C-Mc1}\geq 0.}
\end{equation}
Indeed, by symmetry in $x$ and $y$, we have
\begin{equation*}
\begin{split}
&|\eqref{C-Ma1}|\\
=&4\left|\iint\IIm(\pa_jw\bar w)(y)\left(\frac{\d_{jk}}{|x-y|}-\frac{(x_j-y_j)(x_k-y_k)}{|x-y|^3}\right)\IIm(\pa_kw\bar w)(x){\rm d}x{\rm d}y\right|\\
=&4\left|\iint\frac{1}{|x-y|}\IIm(\not\!\nabla_xw\bar w)(y)\IIm(\not\!\nabla_yw\bar w)(x){\rm d}x{\rm d}y\right|\\
\leq&
2\iint\frac{1}{|x-y|}(|\not\!\nabla_x w(y)|^2|w(x)|^2+|\not\!\nabla_y w(x)|^2|w(y)|^2){\rm d}x{\rm d}y\\
\leq&4\iint\frac{1}{|x-y|}|\not\!\nabla_y w(x)|^2|w(y)|^2{\rm d}x{\rm d}y\\
\leq&
\eqref{C-Mc1},
\end{split}
\end{equation*}
where 
\begin{equation*}
\not\!\nabla_x=\nabla_y-\frac{x-y}{|x-y|}\left(\frac{x-y}{|x-y|}\cdot\nabla_y\right).
\end{equation*}
On the other hand, it is easy to verify that 
\begin{equation}\label{02}
{\eqref{C-Md1}}\geq 0.
\end{equation}
From \eqref{Classical Morawetz 1}, \eqref{01} and \eqref{02}, we obtain that
\begin{equation*}
\begin{split}
\eqref{C-Mb1}
&=\pa_t M(t)-(\eqref{C-Ma1}+\eqref{C-Mc1}+\eqref{C-Md1})-(\eqref{C-Me1}+\eqref{C-Mf1}+\eqref{C-Mg1})\\
&\leq
\pa_t M(t)-(\eqref{C-Me1}+\eqref{C-Mf1}+\eqref{C-Mg1}).
\end{split}
\end{equation*}
Integrating over $[0,T]$, it holds that
\begin{equation*}
\begin{split}
\int_0^T|\eqref{C-Mb1}|{\rm d}t
&\leq
M(T)-M(0)+\int_0^T|\eqref{C-Me1}|+|\eqref{C-Mf1}|+|\eqref{C-Mg1}|{\rm d}t\\
&\lesssim
\|w(t)\|_{L_t^\infty L_x^2}^2\|w(t)\|_{L_t^\infty\dot H_x^{\frac{1}{2}}}^2
+\int_0^T|\eqref{C-Me1}|+|\eqref{C-Mf1}|+|\eqref{C-Mg1}|{\rm d}t.
\end{split}
\end{equation*}
According to Lemma \ref{Lieb-Loss estimate} and Lemma \ref{Visan estimate},
\begin{equation}
\begin{split}
\int_0^T\iint_{\R^5\times\R^5}\frac{|w(t,x)|^2|w(t,y)|^2}{|x-y|^3}{\rm d}x{\rm d}y{\rm d}t
&\overset{\text{Lemma \ref{Lieb-Loss estimate}}}{\simeq}
\big\||\nabla|^{-1}(|w|^2)\big\|^2_{L_t^2([0,T];L_x^2(\R^5))}\\
&\overset{\text{Lemma \ref{Visan estimate}}}{\gtrsim}
\big\||\nabla|^{-\frac{1}{2}}w\big\|_{L_t^4([0,T];L_x^4(\R^5))}^4.
\end{split}
\end{equation}
Hence, 
\begin{equation}\label{e}
\begin{split}
\||\nabla|^{-\frac{1}{2}}w\|_{L_{t,x}^4}^4
&\lesssim
\int_0^T|\eqref{C-Mb1}|{\rm d}t\\
&\lesssim
\|w(t)\|_{L_t^\infty L_x^2}^2\|w(t)\|_{L_t^\infty\dot H_x^{\frac{1}{2}}}^2
+\int_0^T|\eqref{C-Me1}|+|\eqref{C-Mf1}|+|\eqref{C-Mg1}|{\rm d}t.
\end{split}
\end{equation}

Next, we estimate the terms containing \eqref{C-Me1}, \eqref{C-Mf1} and \eqref{C-Mg1}. First, consider term \eqref{C-Me1}. By H\"{o}lder's inequality, we have
\begin{equation}\label{3.6e}
\begin{split}
\int_0^T|\eqref{C-Me1}|{\rm d}t
&\lesssim
\int_0^T\left|\iint\IIm(e\bar w)(y)\frac{x-y}{|x-y|}\cdot\IIm(\nabla w\bar w)(x){\rm d}x{\rm d}y\right|{\rm d}t\\
&\lesssim
\int_0^T\left|\int\IIm(e\bar w)(y){\rm d}y\right|\sup_y\left|\int\frac{x-y}{|x-y|}\cdot\IIm(\nabla w\bar w)(x){\rm d}x\right|{\rm d}t\\
&\lesssim
\|ew\|_{L_t^1L_x^1}\|w\|^2_{L_t^\infty\dot H_x^{\frac{1}{2}}},
\end{split}    
\end{equation}
where $e=(|\cdot|^{-4}\ast|u|^2)u-(|\cdot|^{-4}\ast |w|^2)w$.
By utilising H\"{o}lder's inequality and the Hardy-Littlewood-Sobolev inequality \eqref{Hardy-Littlewood Sobolev inequality formula} for the first term on the right-hand side of \eqref{3.6e}, we obtain that
\begin{equation}\label{NLS-error}
\begin{split}
\|ew\|_{L_t^1L_x^1}
\lesssim&
\|(|\cdot|^{-4}\ast w^2)vw\|_{L_t^1L_x^1}+\|(|\cdot|^{-4}\ast wv)w^2\|_{L_t^1L_x^1}\\
&+\|(|\cdot|^{-4}\ast v^2)w^2\|_{L_t^1L_x^1}+\|(|\cdot|^{-4}\ast wv)vw\|_{L_t^1L_x^1}\\
&+\|(|\cdot|^{-4}\ast v^2)vw\|_{L_t^1L_x^1}\\
\lesssim&
\|v\|_{L_t^2L_x^5}(\|w\|_{L_t^6L_x^3}^3+\|v\|_{L_t^6L_x^3}\|w\|_{L_t^6L_x^3}^2+\|v\|_{L_t^6L_x^3}^2\|w\|_{L_t^6L_x^3})\\
\lesssim&
\|v\|_{L_t^2L_x^5}\|w\|_{L_t^6L_x^3}^3+\|v\|_{L_t^2L_x^5}\|v\|_{L_t^6L_x^3}^3.
\end{split}
\end{equation}
Using the Gagliardo-Nirenberg inequality \eqref{G-N}, 
\begin{equation}
\|w\|_{L_t^6L_x^3}^3
\lesssim
\|w\|_{L_t^\infty H^1_x}\big\||\nabla|^{-\frac{1}{2}}w\big\|_{L_{t,x}^4}^2.
\end{equation}
From above, we obtain that
\begin{equation*}
\begin{split}
\|ew\|_{L_t^1L_x^1}
\lesssim&
\|v\|_{L_t^2L_x^5}\left(\|w\|_{L_t^\infty H^1_x}\big\||\nabla|^{-\frac{1}{2}}w\big\|_{L_{t,x}^4}^2
+\|v\|_{L_t^6L_x^3}^3\right).
\end{split}
\end{equation*}
Hence,
\begin{equation}\label{e1}
\begin{split}
\int_0^T|\eqref{C-Me1}|{\rm d}t
&\lesssim
\|v\|_{L_t^2L_x^5}\left(\|w\|_{L_t^\infty H^1_x}\big\||\nabla|^{-\frac{1}{2}}w\big\|_{L_{t,x}^4}^2
+\|v\|_{L_t^6L_x^3}^3\right)\|w\|^2_{L_t^\infty\dot H_x^{\frac{1}{2}}}.
\end{split}    
\end{equation}
As for \eqref{C-Mf1},
\begin{equation}\label{error1}
\begin{split}
\int_0^T|\eqref{C-Mf1}|{\rm d}t
=
\int_0^T\left|\iint|w(y)|^2\frac{x-y}{|x-y|}\cdot (e\nabla\bar w)(x){\rm d}x{\rm d}y\right|{\rm d}t,
\end{split}
\end{equation}
considering the integration of $x$, we derive that
\begin{equation}\label{error2}
\begin{split}
&\int\frac{x-y}{|x-y|}\cdot \RRe ((|\cdot|^{-4}\ast|u|^2)u-(|\cdot|^{-4}\ast|w|^2)w)\nabla\bar w(t,x){\rm d}x\\
=&
-\int\frac{1}{|x-y|}((|\cdot|^{-4}\ast|u|^2)|u|^2-(|\cdot|^{-4}\ast|w|^2)|w|^2)(t,x){\rm d}x\\
&-\int\frac{x-y}{|x-y|}\cdot\RRe((|\cdot|^{-4}\ast|u|^2)u\nabla\bar v(t,x))
{\rm d}x.
\end{split}
\end{equation}
Therefore, by \eqref{error1}, \eqref{error2}, Hardy's inequality \eqref{Hardy} and H\"{o}lder's inequality, we have
\begin{equation*}
\begin{split}
\int_0^T|\eqref{C-Mf1}|{\rm d}t
\lesssim&
\int_0^T\left|\int((|\cdot|^{-4}\ast|u|^2)|u|^2-(|\cdot|^{-4}\ast|w|^2)|w|^2)(t,x){\rm d}x\right|\|w\|^2_{\dot H^{\frac{1}{2}}_y}{\rm d}t\\
&+
\int_0^T\int\left|(|\cdot|^{-4}\ast|u|^2)u\nabla\bar v(t,x)\right|{\rm d}x\|w\|_{L_y^2}^2{\rm d}t.
\end{split}    
\end{equation*}
Following a similar approach to the proof for estimate \eqref{e1}, the estimate for term \eqref{C-Mf1} can be derived as follows:
\begin{equation}\label{e2}
\begin{split}
\int_0^T|\eqref{C-Mf1}|{\rm d}t
\lesssim&
\|v\|_{L_t^2L_x^5}
\left(\|w\|_{L_t^\infty H^1_x}\big\||\nabla|^{-\frac{1}{2}}w\big\|_{L_{t,x}^4}^2+\|v\|_{L_t^6L_x^3}^3\right)\|w\|^2_{\dot H^{\frac{1}{2}}_y}\\
&+\|\nabla v\|_{L_t^2L_x^5}
\left(\|w\|_{L_t^\infty H^1_x}\big\||\nabla|^{-\frac{1}{2}}w\big\|_{L_{t,x}^4}^2+\|v\|_{L_t^6L_x^3}^3\right)\|w\|_{L_y^2}^2.
\end{split}    
\end{equation}
In a manner analogous to the proof of the estimate \eqref{e1} and \eqref{e2}, we get
\begin{equation}\label{e3}
\begin{split}
\int_0^T|\eqref{C-Mg1}|{\rm d}t
\lesssim
\|v\|_{L_t^2L_x^5}
\left(\|w\|_{L_t^\infty H^1_x}\big\||\nabla|^{-\frac{1}{2}}w\big\|_{L_{t,x}^4}^2+\|v\|_{L_t^6L_x^3}^3\right)\|w\|^2_{\dot H^{\frac{1}{2}}_y}.
\end{split}    
\end{equation}
Combining \eqref{e}, \eqref{e1}, \eqref{e2} and \eqref{e3}, we arrive that
\begin{equation*}
\begin{split}
\||\nabla|^{-\frac{1}{2}}w\|_{L_{t,x}^4}^4
\lesssim
&\|w(t)\|_{L_t^\infty L_x^2}^2\|w(t)\|_{L_t^\infty\dot H_x^{\frac{1}{2}}}^2\\
&+\|v\|_{L_t^2L_x^5}
\left(\|w\|_{L_t^\infty H^1_x}\big\||\nabla|^{-\frac{1}{2}}w\big\|_{L_{t,x}^4}^2+\|v\|_{L_t^6L_x^3}^3\right)\|w\|^2_{L_t^\infty\dot H^{\frac{1}{2}}_y}\\
&+\|\nabla v\|_{L_t^2L_x^5}
\left(\|w\|_{L_t^\infty H^1_x}\big\||\nabla|^{-\frac{1}{2}}w\big\|_{L_{t,x}^4}^2+\|v\|_{L_t^6L_x^3}^3\right)\|w\|_{L_t^\infty L_y^2}^2.
\end{split} 
\end{equation*}
This completes the proof of this proposition.
\end{proof}

\begin{remark}\label{d>5 cannot}
For $d>5$, we consider the term $\|(|\cdot|^{-4}\ast w^2)vw\|_{L_t^1L_x^1}$ in \eqref{NLS-error}. According to H\"{o}lder's inequality and the Hardy-Littlewood-Sobolev inequality, we find 
\begin{equation}\label{d>5 error}
\|(|\cdot|^{-4}\ast w^2)vw\|_{L_t^1L_x^1}
\lesssim
\|v\|_{L_t^{\frac{d+1}{d-2}}L_x^{\frac{2d(d+1)}{d^2-d-8}}}\|w\|_{L_t^{d+1}L_x^{\frac{2(d+1)}{d-1}}}^3,
\end{equation}
where the choice of the norms depends on the  Gagliardo-Nirenberg inequality. Indeed, using Lemma \ref{Lieb-Loss estimate}, Lemma \ref{Visan estimate} and the  Gagliardo-Nirenberg inequality \eqref{G-N}, we obtain
\begin{equation}
\|w\|_{L_t^aL_x^b}\lesssim\|w\|_{L_t^\infty H_x^1}^{1-\theta}\||\nabla|^{-\frac{d-3}{4}}w\|_{L_{t,x}^4}^{\theta}   
\end{equation}
for $0<\theta<1$. Here,
\begin{equation*}
\begin{split}
&\frac{1}{a}=\frac{\theta}{4},\ \frac{1}{b}=\frac{1-\theta}{2}+\frac{\theta}{4},\\    
&(1-\theta)+\theta\left(-\frac{d-3}{4}\right)=0.
\end{split}
\end{equation*}
Hence, we get
\begin{equation}
\theta=\frac{4}{d+1},\ a=d+1,\ b=\frac{2(d+1)}{d-1}.
\end{equation}
However, for $d>5$ it cannot be obtained that \eqref{d>5 error} is bounded in the stochastic sense, since random Strichartz estimate \eqref{random estimate 1} requires $b=\frac{2(d+1)}{d-1}\geq 2$, which implies $d\leq 5.$
\end{remark}

\section{Local well-posedness and stability theories}\label{section 4}

This section is devoted to the proof of the local well-posedness and stability theories for the equation 
\begin{equation}\label{NLS-hartree-w1}
\begin{cases}
i\pa_tw+\Delta w=(|\cdot|^{-4}\ast|w|^2)w+e,\\
w(0,x)=w_0(x),
\end{cases}
\end{equation}
where $e=(|\cdot|^{-4}\ast|u|^2)u-(|\cdot|^{-4}\ast|w|^2)w$, and $w$ satisfies \eqref{v0 w0} and \eqref{v w}. Recall the definitions of $X$-norm:
\begin{equation}\label{X norm}
\|f\|_{X(I)}
\coloneqq
\|\langle\nabla\rangle f\|_{L_t^2(I;L_x^{\frac{10}{3}}(\R^5))}+\|f\|_{L_t^4(I;L_x^4(\R^5))}+\|f\|_{L_t^4(I;L_x^5(\R^5))}+\|f\|_{L_t^3(I;L_x^6(\R^5))},
\end{equation}
and $Y$-norm:
\begin{equation}\label{Y norm}
\begin{split}
\|f\|_{Y(I)}
\coloneqq
&\|\langle\nabla\rangle^{s+\frac{1}{2}a}f\|_{L_t^2(I;L_x^5(\R^5))}+\|f\|_{L_t^4(I;L_x^5(\R^5))}+\|f\|_{L_t^4(I;L_x^4(\R^5))}\\
&+\|f\|_{L_t^6(I;L_x^3(\R^5))}.
\end{split}
\end{equation}


\subsection{Local theory}

\begin{lemma}\label{Local well-posedness}
Let $a\in\mathbb{N}$ satisfy \eqref{parameter a}, $v\in Y(\R)$ and $w_0\in H_x^1(\R^5)$. Then, there exists a unique solution $w$ of \eqref{NLS-hartree-w1} in 
$C([0,T];H_x^1(\R^5))\cap X([0,T]),$
where $T>0$ depending on $a$, $w_0$ and $v$.
\end{lemma}

\begin{remark}
In Lemma \ref{Local well-posedness}, we assumes that $v\in Y(\R)$ and $w_0\in H_x^1(\R^5)$. There is no guarantee that $v\in Y(\R)$ and $w_0\in H_x^1(\R^5)$ are fulfilled for equation \eqref{NLS-hartree-w1} with the initial value $u_0\in H^s_x(\R^5)$ with $s<0$ in the deterministic sense. Nevertheless, they are satisfied in the stochastic sense.
\end{remark}

\begin{proof}[Proof of Lemma \ref{Local well-posedness}]
For simplicity, we always assume that $(t,x)\in[0,T]\times\R^5$. Take the solution map as 
\begin{equation*}
\Phi(w)=e^{it\Delta}w_0-i\int_0^t e^{i(t-s)\Delta}[(|\cdot|^{-4}\ast|u|^2)u(s)]{\rm d}s,
\end{equation*}
which is contraction on the space
\begin{equation}
B_{\delta,T}\coloneqq\{w\in C([0,T];H_x^1):\|w\|_{L_t^\infty H_x^1}\leq2CR,\ \|w\|_{X}\leq 2C\delta\}.
\end{equation}
Here $R>0$ be given such that $\|e^{it\Delta}w_0\|_{L_t^\infty H_x^1\cap X}\leq CR$. We also take the $T>0$ sufficient small such that 
\begin{equation}
\|e^{it\Delta}w_0\|_{X([0,T])}+\|v\|_{Y([0,T])}\leq\delta
\end{equation}
with $\delta>0$ be chosen later.

By the Strichartz estimate,
\begin{equation*}
\begin{split}
\|\Phi(w)\|_{L_t^\infty H_x^1}
&\leq C\|w_0\|_{H_x^1}+C\|\nabla[(|\cdot|^{-4}\ast|u|^2)u]\|_{L_t^1L_x^2}+C\|(|\cdot|^{-4}\ast|u|^2)u\|_{L_t^1L_x^2}.
\end{split}
\end{equation*}
Note that
\begin{equation}\label{proof local w-p 1}
\begin{split}
\|\nabla[(|\cdot|^{-4}\ast|u|^2)u]\|_{L_t^1L_x^2}
&\lesssim
\|(|\cdot|^{-4}\ast(\nabla u \bar u))u\|_{L_t^1L_x^2}+\|(|\cdot|^{-4}\ast|u|^2)\nabla u\|_{L_t^1L_x^2}.
\end{split}
\end{equation}
For the first term of \eqref{proof local w-p 1}, by H\"{o}lder's inequality and the Hardy-Littlewood-Sobolev inequality, we find
\begin{equation*}
\begin{split}
\|(|\cdot|^{-4}\ast&(\nabla u \bar u))u\|_{L_t^1L_x^2}\\
\lesssim&
\||\cdot|^{-4}\ast(\nabla w u)\|_{L_t^\frac{4}{3} L_x^\frac{10}{3}}\|u\|_{L_t^4 L_x^5}+\||\cdot|^{-4}\ast(\nabla v u)\|_{L_t^{\frac{4}{3}} L_x^4}\|u\|_{L_{t,x}^4}\\
\lesssim&
\|\nabla w u\|_{L_t^\frac{4}{3} L_x^2}\|u\|_{L_t^4 L_x^5}+\|\nabla v u\|_{L_t^{\frac{4}{3}} L_x^{\frac{20}{9}}}\|u\|_{L_{t,x}^4}\\
\lesssim&
\|\nabla w\|_{L_t^2 L_x^\frac{10}{3}}\|u\|^2_{L_t^4 L_x^5}+\|\nabla v \|_{L_t^2 L_x^5}\|u\|_{L_{t,x}^4}^2\\
\lesssim&
\|w\|_{X}(\|w\|_{X}+\|v\|_{Y})^2+\|\nabla v \|_{L_t^2 L_x^5}(\|w\|_{X}+\|v\|_{Y})^2.
\end{split}
\end{equation*}
Since $\|\nabla v\|_{L_t^2L_x^5}\lesssim\|v\|_Y\lesssim\delta,$ and $w\in B_{\delta,T}$,
\begin{equation*}
\begin{split}
\|(|\cdot|^{-4}\ast(\nabla u u))u\|_{L_t^1L_x^2}
\lesssim&
\delta^3.
\end{split}
\end{equation*}
The proof for the first term of \eqref{proof local w-p 1} can be followed to get the same bound for the second term of \eqref{proof local w-p 1}, as follows.
\begin{equation*}
\begin{split}
\|(|\cdot|^{-4}\ast|u|^2)\nabla u\|_{L_t^1L_x^2}
\lesssim
\delta^3.
\end{split}
\end{equation*}
Similarly,
\begin{equation*}
\|(|\cdot|^{-4}\ast|u|^2)u\|_{L_t^1L_x^2}
\lesssim
\|w\|_{L_t^2 L_x^\frac{10}{3}}\|u\|^2_{L_t^4 L_x^5}+\|v \|_{L_t^2 L_x^5}\|u\|_{L_{t,x}^4}^2\lesssim\delta^3.
\end{equation*}
Therefore,
\begin{equation*}
\begin{split}
\|\Phi(w)\|_{L_t^\infty H_x^1}
&\leq CR+C\delta^3.
\end{split}
\end{equation*}
Similar as above, we get
\begin{equation*}
\begin{split}
\|\Phi(w)\|_{X}
\leq
&\|e^{it\Delta}w_0\|_{X}+
C\|w\|_{X}(\|w\|_{X}+\|v\|_{L_t^4 L_x^{5}})^2+C\|\nabla v \|_{L_t^2 L_x^5}(\|w\|_{X}+\|v\|_{L_{t,x}^4})^2\\
\leq
&\delta+C\delta^3.
\end{split}
\end{equation*}
Choosing $\delta$ sufficient small such that $\delta^3\leq\min\{R,\delta\}.$ From this, we prove that $\Phi$ maps $B_{\delta,T}$ into itself. By repeating the above steps, we obtain
\begin{equation*}
\begin{split}
\|\Phi(w_1)-\Phi(w_2)\|_{L_t^\infty H_x^1\cap X}
\leq
&C\|\langle\nabla\rangle[(|\cdot|^{-4}\ast|v+w_1|^2)(w_1-w_2)]\|_{L_t^1 L_x^2}\\
&+C\|\langle\nabla\rangle[(|\cdot|^{-4}\ast(w_1-w_2)^2)(v+w_2)]\|_{L_t^1 L_x^2}\\
&+C\|\langle\nabla\rangle[(|\cdot|^{-4}\ast((v+w_2)(w_1-w_2)))(v+w_2)]\|_{L_t^1 L_x^2}\\
\triangleq&
\uppercase\expandafter{\romannumeral 1}_1
+\uppercase\expandafter{\romannumeral 1}_2
+\uppercase\expandafter{\romannumeral 1}_3,
\end{split}
\end{equation*}
where $w_1,w_2\in B_{\delta,T}$. 
For the first term $\uppercase\expandafter{\romannumeral 1}_1:$
\begin{equation}\label{proof local w-p 2}
\begin{split}
\uppercase\expandafter{\romannumeral 1}_1
\leq&
C\|[|\cdot|^{-4}\ast(\langle\nabla\rangle(v+w_1)(v+w_1))](w_1-w_2)\|_{L_t^1L_x^2}\\
&+C\|(|\cdot|^{-4}\ast|v+w_1|^2)\langle\nabla\rangle(w_1-w_2)\|_{L_t^1L_x^2}\\
\triangleq&
\uppercase\expandafter{\romannumeral 2}_1
+\uppercase\expandafter{\romannumeral 2}_2.
\end{split}
\end{equation}
As for $\uppercase\expandafter{\romannumeral 2}_1$, we have
\begin{equation}\label{proof local w-p 3}
\begin{split}
\uppercase\expandafter{\romannumeral 2}_1
\lesssim&
\||\cdot|^{-4}\ast(\langle\nabla\rangle w_1(v+w_1))\|_{L_t^{\frac{4}{3}}L_x^{\frac{10}{3}}}\|w_1-w_2\|_{L_t^4L_x^{5}}\\
&+\||\cdot|^{-4}\ast(\langle\nabla\rangle v(v+w_1))\|_{L_t^{\frac{4}{3}}L_x^4}\|w_1-w_2\|_{L_{t,x}^4}\\
\lesssim&
\|w_1-w_2\|_X(\|\langle\nabla\rangle w_1\|_{L_t^2 L_x^{\frac{10}{3}}}\|v+w_1\|_{L_t^4L_x^5}+\|\langle\nabla\rangle v\|_{L_t^2L_x^5}\|v+w_1\|_{L_{t,x}^4})\\
\lesssim&
\|w_1-w_2\|_X[\|w_1\|_X(\|v\|_Y+\|w_1\|_X)+\|v\|_Y(\|v\|_Y+\|w_1\|_X)]\\
\lesssim&
\delta^2\|w_1-w_2\|_X.
\end{split}
\end{equation}
As for $\uppercase\expandafter{\romannumeral 2}_2$, we obtain that
\begin{equation}\label{proof local w-p 4}
\begin{split}
\uppercase\expandafter{\romannumeral 2}_2
\lesssim&
\|\langle\nabla\rangle(w_1-w_2)\|_{L_t^2L_x^{\frac{10}{3}}}\||\cdot|^{-4}\ast(|v+w_1|^2)\|_{L_t^2L_x^5}\\
\lesssim&
\|\langle\nabla\rangle(w_1-w_2)\|_{L_t^2L_x^{\frac{10}{3}}}\||v+w_1|^2\|_{L_t^2L_x^{\frac{5}{2}}}\\
\lesssim&
\|\langle\nabla\rangle(w_1-w_2)\|_{L_t^2L_x^{\frac{10}{3}}}\|v+w_1\|^2_{L_t^4L_x^5}\\
\lesssim&
\|w_1-w_2\|_X(\|v\|_Y+\|w\|_X)^2\\
\lesssim&
\delta^2\|w_1-w_2\|_X.
\end{split}
\end{equation}
Combining estimates \eqref{proof local w-p 2}, \eqref{proof local w-p 3} and \eqref{proof local w-p 4}, we find that
\begin{equation*}
\uppercase\expandafter{\romannumeral 1}_1
\leq C\delta^2\|w_1-w_2\|_X.
\end{equation*}
We also derive that
\begin{equation*}
\uppercase\expandafter{\romannumeral 1}_2
\leq C\delta^2\|w_1-w_2\|_X,
\quad\text{and}\quad
\uppercase\expandafter{\romannumeral 1}_3
\leq C\delta^2\|w_1-w_2\|_X.
\end{equation*}
Hence,
\begin{equation*}
\begin{split}
\|\Phi(w_1)-\Phi(w_2)\|_{L_t^\infty H_x^1\cap X}
\leq
C\delta^2\|w_1-w_2\|_X
\leq C\delta^2\|w_1-w_2\|_{L_t^\infty H_x^1\cap X}.
\end{split}
\end{equation*}
Choosing $\delta$ sufficient small such $C\delta^2<\frac{1}{2}$, which shows that $\Phi$ is a contraction mapping on $B_{\delta,T}.$ This completes the proof of Lemma \ref{Local well-posedness}.
\end{proof}

\subsection{Stability theories} 
Finally, the stability results for equation \eqref{NLS-hartree-w1} are given in the following.



\begin{lemma}[Short time stability]\label{short stability}
Let $a\in\mathbb{N}$ satisfy \eqref{parameter a}, $w_0\in H_x^1(\R^5)$, $I\subseteq\R$ be a compact time interval, and $0\in I$. Let $w$ be the solution defined on $I\times\R^5$ of the perturbed equation \eqref{NLS-hartree-w1}.
Let $\tilde w_0\in H_x^1(\R^5)$ and let $\tilde w$ be the solution of the defocusing  energy-critical equation on $I\times\R^5$:
\begin{equation}\label{tilde w}
\begin{cases}
i\pa_t\tilde w+\Delta\tilde w=(|\cdot|^{-4}\ast|\tilde w|^2)\tilde w,\\
\tilde w(0,x)=\tilde w_0.
\end{cases}
\end{equation}

Then, for $\e_0>0$ sufficiently small such that if $0<\e\leq\e_0$ and 
\begin{align}
\|\tilde w\|_{X(I)}\leq \e_0,\  
\|w_0-\tilde w_0\|_{H_x^1(\R^5)}&\leq\e_0,\label{short stability1} \\
\|v\|_{Y(I)}&\leq\e, \label{short stability2}
\end{align}
there exists $C\geq 0$ such that
\begin{equation}\label{short stability result}
\|w-\tilde w\|_{L_t^\infty(I;H_x^1(\R^5))}+\|w-\tilde w\|_{X(I)}\leq C\e.
\end{equation}
\end{lemma}

\begin{proof}
Set $g(t,x)\coloneqq w(t,x)-\tilde w(t,x)$. Then $g$ satisfies
\begin{equation}
\begin{cases}
i\pa_t g+\Delta g=F(g+\tilde w+v)-F(\tilde w),\\
g(0)=w_0-\tilde w_0,
\end{cases}
\end{equation}
where $F(g)=(|\cdot|^{-4}\ast|g|^2)g$.
To prove the inequality \eqref{short stability result}, it is sufficient to prove 
\begin{equation}
\|g\|_{L_t^\infty(I;H_x^1(\R^5))}+\|g\|_{X(I)}
\leq C\e.
\end{equation}
By the Strichartz estimates and the Duhamel formula
\begin{equation*}
g(t,x)=e^{it\Delta}g(0)-i\int_0^t e^{i(t-s)\Delta}\left(F(g+\tilde w+v)-F(\tilde w)\right)(s){\rm d}s,
\end{equation*}
we obtain that
\begin{equation*}
\|g\|_{L_t^\infty H_x^1}+\|g\|_{X}
\lesssim
\|g(0)\|_{H_x^1}+\|\langle\nabla\rangle\left[F(g+\tilde w+v)-F(\tilde w)\right]\|_{L_t^1L_x^2}.
\end{equation*}
Following the argument as in Lemma \ref{Local well-posedness}, and using estimates \eqref{short stability1} and \eqref{short stability2}, we obtain
\begin{equation*}
\begin{split}
\|g\|_{L_t^\infty H_x^1}+\|g\|_{X}
&\lesssim
\|g(0)\|_{H_x^1}+\|g\|_X^3+\|g\|_X(\|v\|_Y^2+\|\tilde w\|_X^2)+\|v\|_Y^3+\|v\|_Y\|\tilde w\|_X^2\\
&\lesssim
(\e_0+\e\e_0^2+\e^3)+\e_0^2\|g\|_X+\|g\|_X^3.
\end{split}
\end{equation*}
Choosing $\e_0$ sufficiently small, the continuity argument shows that $\|g\|_X\lesssim\e$ for all $t\in I$. And we obtain that
\begin{equation*}
\|g\|_{L_t^\infty(I;H_x^1(\R^5))}+\|g\|_{X(I)}
\leq C\e,
\end{equation*}
which completes the proof of this lemma.
\end{proof}

\begin{lemma}[Long time stability]\label{Long stability}
Let $a\in\mathbb{N}$ satisfy \eqref{parameter a},  $w_0\in H_x^1(\R^5)$, $I\subseteq\R$ be a compact time interval, $0\in I$ and $M>0$. Let $w$ be a solution defined on $I\times\R^5$ of the perturbed equation \eqref{NLS-hartree-w1}.
Let $\tilde w_0\in H_x^1(\R^5)$ and let $\tilde w$ be the solution of the defocusing energy-critical equation \eqref{tilde w}

Then, for $\e>0$ sufficiently small, $M>0$ such that 
\begin{align} \label{Long stability1}
\|\tilde w\|_{X(I)}\leq M,
\|w_0-\tilde w_0\|_{H_x^1(\R^5)}&\leq\e,
\end{align}
and there exists $\e_1>0$ such that $v$ satisfies
\begin{align}
\|v\|_{Y(I)}&\leq\e_1, \label{Long stability2}
\end{align}
then the following holds:
\begin{equation}\label{long stability result}
\|w-\tilde w\|_{L_t^\infty(I;H_x^1(\R^5))}+\|w-\tilde w\|_{X(I)}\leq C_M\e_1,
\end{equation}
where $C_M$ is the constant associated with $M$.
\end{lemma}
\begin{proof}
Choose $\e_0$ as in Lemma \ref{short stability} and divide $I$ into $J=J(M,\e_0)$ subintervals $I_j=[t_{j-1},t_j)$ such that
\begin{equation*}
\frac{1}{2}\e_0\leq\|\tilde w\|_{X(I_j)}\leq\e_0
\end{equation*}
for each $j=1,\cdots,J$. Let $0<\e_1<\e_0$, Lemma \ref{short stability} shows that
\begin{equation*}
\|w-\tilde w\|_{L_t^\infty(I_1;H_x^1(\R^5))}+\|w-\tilde w\|_{X(I_1)}\leq C\e_1
\end{equation*}
on the first interval $I_1$. In particular,
\begin{equation*}
\|w(t_1)-\tilde w(t_1)\|_{H_x^1(\R^5)}\leq C\e_1.
\end{equation*}
Choosing $\e_1$ sufficiently small such that $C\e_1<\e_0$, we can apply Lemma \ref{short stability} on the interval $I_2$ and obtain
\begin{equation*}
\|w-\tilde w\|_{L_t^\infty(I_2;H_x^1(\R^5))}+\|w-\tilde w\|_{X(I_2)}\leq C^2\e_1.    
\end{equation*}
Recursively, we obtain
\begin{equation*}
\|w-\tilde w\|_{L_t^\infty(I_j;H_x^1(\R^5))}+\|w-\tilde w\|_{X(I_j)}\leq C^j\e_1,
\end{equation*}
for each $j=1,\cdots,J$, as long as $\max_{j=1,\cdots,J} C^j\e_1<\e_0$. We complete the proof of this lemma by using the finiteness of $J=J(M,\e)$ and choosing $\e_1\leq\frac{\e_0}{C^{J(M,\e)}}$
\end{proof}

Before proceeding to the proof of Theorem \ref{deterministic problem}, we give the global space-time bound for the solution of equation \eqref{tilde w}

\begin{lemma}[Global space-time bound]\label{global space-time bound}
For any $\tilde w_0$ with finite energy, i.e. 
\begin{equation*}
E(\tilde w_0)=\frac{1}{2}\int_{\R^5}|\nabla\tilde w(0,x)|^2{\rm d}x
+\frac{1}{4}\iint_{\R^5\times\R^5}\frac{|\tilde w(0,x)|^2|\tilde w(0,y)|^2}{|x-y|^4}{\rm d}x{\rm d}y<\infty,
\end{equation*}
there exists a unique global solution $\tilde w\in C(\R;H_x^1(\R^5))$ to equation \eqref{tilde w} with the initial data $\tilde w_0\in H_x^1(\R^5)$, such that
\begin{equation}
\|\tilde w\|_{X(\R)}\leq C(\|\tilde w_0\|_{H_x^1(\R^5)})
\end{equation}
for some constant $C(\|\tilde w_0\|_{H_x^1(\R^5)})$.
\end{lemma}
\begin{proof}
The proof is based on the results in \cite{Miao-Xu-Zhao 2011(PDE)}. From \cite{Miao-Xu-Zhao 2011(PDE)}, we get that there exists a unique global solution $\tilde w\in C(\R;H_x^1)\cap L_t^6(\R;L_x^{\frac{30}{7}}(\R^5))$ to \eqref{tilde w} such that
\begin{equation}
\|\tilde w\|_{L_t^6(\R;L_x^{\frac{30}{7}}(\R^5))}\leq C(E(\tilde w_0)).
\end{equation}
Then, divide $\R$ into $J=J(C(E(\tilde w_0),\eta)$ subintervals $I_j=[t_{j-1},t_j),j=1,\cdots,J,$ such that
\begin{equation*}
\|\tilde w\|_{L_t^6(I_j;L_x^{\frac{30}{7}})}\leq\eta
\end{equation*}
for some small $\eta>0$ to be chosen later. By the Duhamel formula, we have
\begin{equation*}
\tilde w(t,x)=e^{i(t-t_{j-1})\Delta}\tilde w(t_{j-1})-i\int_{t_{j-1}}^t e^{i(t-s)\Delta}(|\cdot|^{-4}\ast|\tilde w|^2)\tilde w(s){\rm d}s.
\end{equation*}
Using Lemma \ref{classical Strichatz}, we find
\begin{equation*}
\begin{split}
\|\tilde w\|_{X(I_j)}
\lesssim&
\|\tilde w(t_{j-1})\|_{H_x^1}+\|\langle\nabla\rangle[(|\cdot|^{-4}\ast|\tilde w|^2)\tilde w]\|_{L_t^1(I_j;L_x^2)}\\
\lesssim&
\|\tilde w(t_{j-1})\|_{H_x^1}+\|\langle\nabla\rangle\tilde w\|_{L_t^2(I_j;L_x^{\frac{10}{3}})}\|\tilde w\|_{L_t^6(I_j;L_x^{\frac{30}{7}})}\|\tilde w\|_{L_t^3(I_j;L_x^{6})}\\
\lesssim&
\|\tilde w(t_{j-1})\|_{H_x^1}+\|\tilde w\|_{L_t^6(I_j;L_x^{\frac{30}{7}})}\|\tilde w\|_{X(I_j)}^2.
\end{split}
\end{equation*}
First, estimate $\tilde w$ on the interval $I_1$:
\begin{equation*}
\begin{split}
\|\tilde w\|_{X(I_1)}
\lesssim&
\|\tilde w_0\|_{H_x^1}+\|\tilde w\|_{L_t^6(I_1;L_x^{\frac{30}{7}})}\|\tilde w\|_{X(I_1)}^2.
\end{split}
\end{equation*}
Let $\|\tilde w\|_{L_t^6(I_1;L_x^{\frac{30}{7}})}<4^{-2}\|\tilde w_0\|_{H_x^1}^{-1}$, by the standard continuity argument, we yield that
\begin{equation*}
\|\tilde w\|_{X(I_1)}\leq 2\|\tilde w_0\|_{H_x^1}.
\end{equation*}
In particular, $\|w(t_1)\|\leq 2\|\tilde w_0\|_{H_x^1}$.
Similarly as above, estimating $\tilde w$ on the interval $I_2:$
\begin{equation*}
\begin{split}
\|\tilde w\|_{X(I_2)}
\lesssim&
\|\tilde w(t_1)\|_{H_x^1}+\|\tilde w\|_{L_t^6(I_2;L_x^{\frac{30}{7}})}\|\tilde w\|_{X(I_2)}^2\\
\lesssim&
2\|\tilde w_0\|_{H_x^1}+\|\tilde w\|_{L_t^6(I_2;L_x^{\frac{30}{7}})}\|\tilde w\|_{X(I_2)}^2.
\end{split}
\end{equation*}
Let  
$\|\tilde w\|_{L_t^6(I_2;L_x^{\frac{30}{7}})}<4^{-2}(2\|\tilde w_0\|_{H_x^1})^{-1}$, by the standard continuity argument, we obtain that
\begin{equation*}
\|\tilde w\|_{X(I_2)}\leq 2^2\|\tilde w_0\|_{H_x^1}.
\end{equation*}
Arguing recursively, we get
\begin{equation*}
\|\tilde w\|_{X(I_j)}\leq 2^j\|\tilde w_0\|_{H_x^1}.
\end{equation*}
for each $j=1,\cdots,J$, as long as 
\begin{equation*}
\max_{j=1,\cdots,J}\|\tilde w\|_{L_t^6(I_j;L_x^{\frac{30}{7}})}<4^{-2}(2^{J-1}\|\tilde w_0\|_{H_x^1})^{-1},
\end{equation*}
i.e. $\eta<4^{-2}(2^{J-1}\|\tilde w_0\|_{H_x^1})^{-1}.$ The proof of this lemma is completed.
\end{proof}

\section{Proof of Theorem \ref{deterministic problem}}\label{section 5}
In this section, we prove Theorem \ref{deterministic problem}. Its proof is based on the almost conservation laws (Proposition \ref{almost conservation law}), the interaction Morawetz estimate and the stability theories. Recall the perturbed equation:
\begin{equation}\label{NLS-hartree-w2}
\begin{cases}
i\pa_t w+\Delta w=(|\cdot|^{-4}\ast|w|^2)w+e,\\
w(0,x)=w_0(x),
\end{cases}
e=(|\cdot|^{-4}\ast|u|^2)u-(|\cdot|^{-4}\ast|w|^2)w,
\end{equation}
and the definitions of energy and mass for the solution $w$ of \eqref{NLS-hartree-w2}:
\begin{equation}\label{NLS-w energy}
E_w(t)\coloneqq\frac{1}{2}\int_{\R^5}|\nabla w(t,x)|^2{\rm d}x+\frac{1}{4}\iint_{\R^5\times\R^5}\frac{|u(t,x)|^2|u(t,y)|^2}{|x-y|^4}{\rm d}x{\rm d}y,
\end{equation}
\begin{equation}\label{NLS-w mass}
M_w(t)\coloneqq\int_{\R^5}|w(t,x)|^2{\rm d}x.
\end{equation}
We prove that the mass and energy of the solution $w$ of \eqref{NLS-hartree-w2} are almost conserved.  To prove these, we introduce the following lemma, which is an ingredient for the later proof of Proposition \ref{almost conservation law}. 

\begin{lemma}\label{mass,energy,d/dt,bound}
Assume that $w\in C(I;H_x^1(\R^5))$ solves \eqref{NLS-hartree-w2}. Let $E_w(t)$ and $M_w(t)$ be defined as \eqref{NLS-w energy} and \eqref{NLS-w mass}. Then for any $t\in I$,
\begin{align}
\left|\frac{\rm d}{{\rm d}t}M_w(t)\right|
&\leq
2\left|\int_{\R^5}((|\cdot|^{-4}\ast|u|^2)u-(|\cdot|^{-4}\ast|w|^2)w)\bar w{\rm d}x\right|,\label{mass,d/dt,bound}\\
\left|\frac{\rm d}{{\rm d}t}E_w(t)\right|
&\leq
\left|\int_{\R^5}(|\cdot|^{-4}\ast|u|^2)u\Delta\bar v{\rm d}x\right|\label{energy,d/dt,bound}.
\end{align}
\end{lemma}
\begin{proof}
By integration by parts, we have
\begin{equation*}
\frac{\rm d}{{\rm d}t}M_w(t)
=\frac{\rm d}{{\rm d}t}\left(\int_{\R^5}|w(t,x)|^2{\rm d}x\right)
=2\RRe\int_{\R^5}\pa_t w\bar w{\rm d}x.
\end{equation*}
According to \eqref{NLS-hartree-w2}, 
\begin{equation*}
\pa_t w=-i\left[(|\cdot|^{-4}\ast|u|^2)u-\Delta w\right].
\end{equation*}
Hence,
\begin{equation*}
\begin{split}
\frac{\rm d}{{\rm d}t}M_w(t)
&=2\RRe\int_{\R^5}\bar w\left[-i\left((|\cdot|^{-4}\ast|u|^2)u-\Delta w\right)\right]{\rm d}x\\
&=-2\RRe\int_{\R^5}\bar w\left(i(|\cdot|^{-4}\ast|u|^2)u\right)+i\bar w\Delta w{\rm d}x\\
&=-2\RRe\int_{\R^5}\bar w\left(i(|\cdot|^{-4}\ast|u|^2)u\right){\rm d}x\\
&=-2\RRe\int_{\R^5}\bar w\left[\left(i(|\cdot|^{-4}\ast|u|^2)u\right)-i(|\cdot|^{-4}\ast|w|^2)w\right]{\rm d}x,
\end{split}
\end{equation*}
which gives the inequality \eqref{mass,d/dt,bound}.
By integration by parts, we derive that
\begin{equation}\label{E(w(t))1}
\begin{split}
\frac{\rm d}{{\rm d}t}\left(\frac{1}{2}|\nabla w(t,x)|^2{\rm d}x\right)
&=\RRe\int_{\R^5}\nabla w\cdot\nabla \pa_t\bar w{\rm d}x\\
&=-\RRe\int_{\R^5}\Delta w\pa_t\bar w{\rm d}x\\
&=-\RRe\int_{\R^5}\left[(|\cdot|^{-4}\ast|u|^2)u-i\pa_t w\right]\pa_t\bar w{\rm d}x\\
&=-\RRe\int_{\R^5}(|\cdot|^{-4}\ast|u|^2)u\pa_t \bar w{\rm d}x\\
&=-\RRe\int_{\R^5}(|\cdot|^{-4}\ast|u|^2)u(\pa_t\bar u-\pa_t\bar v){\rm d}x\\
&=-\frac{1}{4}\frac{\rm d}{{\rm d}t}\left(\int_{\R^5}(|\cdot|^{-4}\ast|u|^2){\rm d}x\right)+\RRe\int_{\R^5}(|\cdot|^{-4}\ast|u|^2)u\pa_t\bar v{\rm d}x.
\end{split}
\end{equation}
Combining \eqref{E(w(t))1} and $\pa_t\bar v=-i\Delta\bar v$, we arrive at \eqref{energy,d/dt,bound}.
The proof of this lemma is completed.
\end{proof}

Subsequently, one of the most important tools for proving Theorem \ref{deterministic problem} is presented: the almost conservation laws.
\begin{proposition}\label{almost conservation law}
Let $a\in\mathbb{N}$ satisfy \eqref{parameter a}, $s<0$ and $A>0$. Suppose $w$ is the solution of \eqref{NLS-hartree-w2} and $v=e^{it\Delta}v_0\in Y(\R)$ satisfies \eqref{v0 w0}. Take $T>0$ such that $w\in C([0,T];H_x^1(\R^5))$. Suppose there exists $N_0=N_0(A)$ with the following properties:
\begin{equation*}
\|u_0\|_{H_x^s(\R^5)}+\|v\|_{Y(\R)}\leq A
\end{equation*}
and
\begin{equation*}
M_w(0)\leq AN_0^{-2s},\  E_w(0)\leq AN_0^{2(1-s)}.
\end{equation*}
Then, we have
\begin{equation}
\sup_{t\in[0,T]}M_w(t)\leq 2AN_0^{-2s},\  \sup_{t\in[0,T]}E_w(t)\leq 2AN_0^{2(1-s)}.
\end{equation}
\end{proposition}
\begin{proof}
By a bootstrap continuity argument, it suffices to prove that if 
\begin{equation}
\sup_{t\in I}M_w(t)\leq 2AN_0^{-2s},\  \sup_{t\in I}E_w(t)\leq 2AN_0^{2(1-s)},
\end{equation}
then
\begin{equation}
\sup_{t\in I}M_w(t)\leq \frac{3}{2}AN_0^{-2s},\  \sup_{t\in I}E_w(t)\leq \frac{3}{2}AN_0^{2(1-s)},
\end{equation}
where $I\subset[0,T]$.

Since $\hat v=\mathcal{F}(e^{it\Delta}P_{\geq N_0}u_0(x))$ is supported on the high frequency region, we deduce that for $0\leq l\leq s+\frac{a}{2}$,
\begin{equation}\label{p-conservation 1}
\||\nabla|^lv\|_{L_t^2L_x^5}
\lesssim
N_0^{l-s-\frac{a}{2}}\|v\|_{Y}
\lesssim
N_0^{l-s-\frac{a}{2}}
\lesssim1.
\end{equation}
By bootstrap hypothesis, we have
\begin{equation}\label{p-conservation 2}
\|w\|_{L_t^\infty L_x^2}\lesssim N_0^{-s},\quad\|w\|_{L_t^\infty \dot H_x^1}\lesssim N_0^{1-s}.
\end{equation}
Combined with the interpolation, this leads to
\begin{equation}\label{p-conservation 3}
\|w\|_{L_t^\infty \dot H_x^\alpha}\lesssim N_0^{\alpha-s},
\end{equation}
for any $0\leq\alpha\leq 1.$ 
Using Proposition \ref{morawetz NLS-hartree-w}, inequalities \eqref{p-conservation 1}, \ref{p-conservation 2} and \ref{p-conservation 3}, we derive that
\begin{equation*}
\begin{split}
\big\||\nabla|^{-\frac{1}{2}}w\big\|_{L_{t,x}^4}^4
&\lesssim
N_0^{1-4s}+N_0^{2-3s-\frac{a}{2}}\left(N_0^{1-s}\big\||\nabla|^{-\frac{1}{2}}w\big\|_{L_{t,x}^4}^2+1\right)\\
&\lesssim
N_0^{1-4s}+N_0^{2-3s-\frac{a}{2}}+N_0^{6-8s-a}+\frac{1}{2}\big\||\nabla|^{-\frac{1}{2}}w\big\|_{L_{t,x}^4}^4.
\end{split}
\end{equation*}
Since $s>\frac{5}{4}-\frac{1}{4}a$, we get
\begin{equation}
\big\||\nabla|^{-\frac{1}{2}}w\big\|_{L_{t,x}^4}^4 
\lesssim
N_0^{1-4s}.
\end{equation}
For the mass bound, combining Lemma \ref{mass,energy,d/dt,bound}, H\"{o}lder's inequality, \eqref{p-conservation 1} and \eqref{p-conservation 2} with the proof steps of \eqref{NLS-error}, we obtain
\begin{equation}
\begin{split}
\int_{I}\text{RHS of \eqref{mass,d/dt,bound}}{\rm d}t
=&
\int_I\left|\frac{\rm d}{{\rm d}t}M_w(t)\right|{\rm d}t\\
\lesssim&
\int_I\left|\int_{\R^5}((|\cdot|^{-4}\ast|u|^2)u-(|\cdot|^{-4}\ast|w|^2)w)\bar w{\rm d}x\right|{\rm d}t\\
\lesssim&
\|v\|_{L_t^2L_x^5}\left(\|w\|_{L_x^\infty H^1_x}\big\||\nabla|^{-\frac{1}{2}}w\big\|_{L_{t,x}^4}^2
+\|v\|_{L_t^6L_x^3}^3\right)\\
\lesssim&
N_0^{\frac{3}{2}-4s-\frac{a}{2}}.
\end{split}
\end{equation}
Using integration by parts and $s>\frac{5}{4}-\frac{1}{4}a$, we deduce that
\begin{equation*}
\begin{split}
M_w(t)
&\leq M_w(0)+\int_I\left|\frac{\rm d}{{\rm d}t}M_w(t)\right|{\rm d}t\\
&\leq AN_0^{-2s}+C(A)N_0^{\frac{3}{2}-4s-\frac{a}{2}}\\
&\leq \frac{3}{2}AN_0^{-2s},
\end{split}
\end{equation*}
where $C(A)N_0^{\frac{3}{2}-2s-\frac{a}{2}}\leq\frac{1}{2}A.$

Similar to the proof above and \eqref{e2}, we get the energy bound:
\begin{equation}
\begin{split}
\int_{I}\text{RHS of \eqref{energy,d/dt,bound}}{\rm d}t
=&
\int_I\left|\frac{\rm d}{{\rm d}t}E_w(t)\right|{\rm d}t\\
\lesssim&
\int_I\left|\int_{\R^5}(|\cdot|^{-4}\ast|u|^2)u\Delta\bar v{\rm d}x\right|{\rm d}t\\
\lesssim&
\|\Delta v\|_{L_t^2L_x^5}\left(\|w\|_{L_x^\infty H^1_x}\big\||\nabla|^{-\frac{1}{2}}w\big\|_{L_{t,x}^4}^2+\|v\|_{L_t^6L_x^3}^3\right)
\\
\lesssim&
N_0^{\frac{7}{2}-4s-\frac{a}{2}}.
\end{split}
\end{equation}
Using integration by parts and $s>\frac{5}{4}-\frac{1}{4}a$, we have
\begin{equation*}
\begin{split}
E_w(t)
&\leq E_w(0)+\int_I\left|\frac{\rm d}{{\rm d}t}E_w(t)\right|{\rm d}t\\
&\leq AN_0^{2(1-s)}+C(A)N_0^{\frac{7}{2}-4s-\frac{a}{2}}\\
&\leq \frac{3}{2}AN_0^{2(1-s)}.
\end{split}
\end{equation*}
This completes the proof of the proposition.
\end{proof}

Now, we are in position to show Theorem \ref{deterministic problem}.
\begin{proof}[Proof of Theorem \ref{deterministic problem}]
We only consider the forward time interval $[0,+\infty)$ and the other half is similar. By Proposition \ref{almost conservation law}, we obtain that if $w\in C([0,T];H_x^1)$ solves \eqref{NLS-hartree-w2} for some $T>0,$ then
\begin{equation}\label{determinstic th es1}
\sup_{t\in[0,T]}\|w(t)\|_{H^1_x}
\leq 2\sup_{t\in[0,T]}E_w(t)+\sup_{t\in[0,T]}M_w(t)
\leq 6AN_0^{2(1-s)}.
\end{equation}
Given $t'\in[0,\infty)$, consider equation \eqref{tilde w} with the initial data $w(t',x)$:
\begin{equation}
\begin{cases}
i\pa_t\tilde w+\Delta\tilde w=(|x|^{-4}\ast|\tilde w|^2)\tilde w,\\
\tilde w(t',x)=w(t',x).
\end{cases}
\end{equation}
From Lemma \ref{global space-time bound}, there exists a unique solution $\tilde w^{(t')}(t,x)$ satisfying
\begin{equation}\label{determinstic th es2}
\|\tilde w^{(t')}\|_{L_t^\infty(\R;H_x^1)}+\|\tilde w^{(t')}\|_{X(\R)}\leq C(\|w(t')\|_{H_x^1}),
\end{equation}
where $C(\|w(t')\|_{H_x^1})$ is a constant. According to Lemma \ref{Local well-posedness}, there exists a unique solution $w$ of \eqref{NLS-hartree-w2} in $C([0,T];H_x^1)\cap X([0,T])$ for some $T>0$. Then, for $t'\in[0,T]$, using \eqref{determinstic th es1} and \eqref{determinstic th es2}, we have
\begin{equation}\label{estimate 1}
\sup_{t'\in[0,T]}(\|\tilde w^{(t')}\|_{L_t^\infty(\R;H_x^1)}+\|\tilde w^{(t')}\|_{X(\R)})\leq C(6AN_0^{2(1-s)})\eqqcolon C(A,N_0).
\end{equation}
Since $\|v\|_{Y(\R)}<A$, we can split $[0,\infty)=\cup_{k=1}^K I_k=[t_{k-1},t_k)$ with $t_0=0$, such that
\begin{equation*}
\|v\|_{Y(I_k)}\sim\e_1,
\end{equation*}
where $\e_1=\e_1(A,N_0)$ is define in Lemma \ref{Long stability} and $K=K(A,N_0,\e_1).$ From \eqref{estimate 1},we get
\begin{equation}
\|\tilde w^{(t_0)}\|_{L_t^\infty(\R;H_x^1)}+\|\tilde w^{(t_0)}\|_{X(\R)}\leq C(A,N_0).
\end{equation}
Using Lemma \ref{Long stability} on the first interval $I_1$, there exists a solution $w\in C(I_1;H_x^1)$ satisfies
\begin{equation}
\|w-\tilde w^{(t_0)}\|_{L_t^\infty(I_1;H_x^1)}+\|w-\tilde w^{(t_0)}\|_{X(I_1)}\leq C\e_1.
\end{equation}
Then by the triangle inequality,
\begin{equation*}
\begin{split}
\|w\|_{L_t^\infty(I_1;H_x^1)}+\|w\|_{X(I_1)}
&\leq C\e_1+\|\tilde w^{(t_0)}\|_{L_t^\infty(I_1;H_x^1)}+\|\tilde w^{(t_0)}\|_{X(I_1)}\\
&\leq C\e_1+C(A,N_0). 
\end{split}
\end{equation*}
In particular, for the solution $\tilde w^{(t_1)}$ of the equation
\begin{equation*}
\begin{cases}
i\pa_t\tilde w+\Delta\tilde w=(|x|^{-4}\ast|\tilde w|^2)\tilde w,\\
\tilde w(t_1,x)=w(t_1,x),
\end{cases}
\end{equation*}
we have 
\begin{equation*}
\|\tilde w^{(t_1)}(t_1)\|_{H_x^1}=\|w(t_1)\|_{H_x^1}
\leq
C\e_1+C(A,N_0).
\end{equation*}
Repeating the above step on $I_2$, we can obtain the existence of $w\in C(I_2;H_x^1)$ satisfies
\begin{equation*}
\begin{split}
\|w\|_{L_t^\infty(I_2;H_x^1)}+\|w\|_{X(I_2)}
\leq 2C\e_1+C(A,N_0).  
\end{split}
\end{equation*}
Applying the previous argument inductively for all $I_k$, $k=1,\cdots,K$, we obtain that there exists $w\in C(I_k;H_x^1)$ satisfies
\begin{equation*}
\begin{split}
\|w\|_{L_t^\infty(I_k;H_x^1)}+\|w\|_{X(I_k)}
\leq kC\e_1+C(A,N_0).
\end{split}
\end{equation*}
Summing up these bound, we deduce that $w\in C([0,\infty);H_x^1)$ and 
\begin{equation}\label{estimate 2}
\begin{split}
\|w\|_{L_t^\infty([0,\infty);H_x^1)}+\|w\|_{X([0,\infty))}
\leq C\e_1\sum_{k=1}^Kk+C(A,N_0)\eqqcolon C(A,N_0,\e_1).
\end{split}
\end{equation}
Hence we get the global well-posedness of $w$ and consequently of $u=w+v$. 

Then, the scattering of $u$ follows from \eqref{estimate 2} and a standard application of the Strichartz estimate. Similar to the proof of \emph{Proposition 4.1} in \cite{Shen-Soffer-Wu 1}, it suffices to prove that
\begin{equation}
\left\|\langle\nabla\rangle\int_0^\infty e^{-is\Delta}(|\cdot|^{-4}\ast|u|^2)u(s){\rm d}s\right\|_{L_x^2}\leq C(A).
\end{equation}
Using the argument in Lemma \ref{Local well-posedness} and the assumption of $v$ in Theorem \ref{deterministic problem},
\begin{equation*}
\begin{split}
&\left\|\langle\nabla\rangle\int_0^\infty e^{-is\Delta}(|\cdot|^{-4}\ast|u|^2)u(s){\rm d}s\right\|_{L_x^2}\\
\lesssim&
\|\langle\nabla\rangle[(|\cdot|^{-4}\ast|u|^2)u]\|_{L_t^1L_x^2}\\
\lesssim&
\|\langle\nabla\rangle w\|_{L_t^2L_x^{\frac{10}{3}}}\|u\|_{L_t^4L_x^5}^2+\|\langle\nabla\rangle v\|_{L_t^2L_x^5}\|u\|^2_{L_{t,x}^4}\\
\lesssim&
C(A),
\end{split}
\end{equation*}
which completes the proof of this theorem.
\end{proof}

\section{Proof of Theorem \ref{main theorem}}\label{section 6}
This section aims to prove Theorem \ref{main theorem}. First, we prove that the condition \eqref{deterministic Condition1} holds under the random sense in Proposition \ref{Y random bound}. More precisely, there exists $\tilde\Omega\subseteq\Omega$, $A\geq 0$ and $N_0=N_0(A)\gg1$, such that for any $\omega\in\tilde\Omega$, 
\begin{equation*}
\|u_0^\omega\|_{H^s_x(\R^5)}+\|v^\omega\|_{Y(\R)}\leq A.
\end{equation*}
Then, based on Theorem \ref{deterministic problem} and Proposition \ref{Y random bound}, we give the proof of Theorem \ref{main theorem}.

\begin{proposition}\label{Y random bound}
Let $s\in\R$ and $f\in H_x^s$. Define the randomization $f^\omega$ as in Definition \ref{randomization}. Then there exists absolute constants $C>0$ and $c>0$ such that for any $\lambda>0$ it holds that
\begin{equation}
\mathbb{P}\left(\{\omega\in\Omega:\|e^{it\Delta}f^\omega\|_{Y(\R)}>\lambda\}\right)\leq C\exp\{-c\lambda^2\|f\|_{H_x^s(\R^5)}^2\}.
\end{equation}
In particular, we have for almost every $\omega\in\Omega$
\begin{equation}
\begin{split}
\|e^{it\Delta}f^\omega\|_{Y(\R)}
=\|\langle\nabla\rangle&^{s+\frac{1}{2}a}e^{it\Delta}f^\omega\|_{L_t^2(\R;L_x^5(\R^5))}+\|e^{it\Delta}f^\omega\|_{L_t^4(\R;L_x^5(\R^5))}\\
&+\|e^{it\Delta}f^\omega\|_{L_t^4(\R;L_x^4(\R^5))}+\|e^{it\Delta}f^\omega\|_{L_t^6(\R;L_x^3(\R^5))}
<\infty,  
\end{split}
\end{equation}
i.e.
\begin{equation}
\mathbb{P}\left(\{\omega\in\Omega:\|e^{it\Delta}f^\omega\|_{Y(\R)}<\infty\}\right)=1.
\end{equation}
\end{proposition}
\begin{proof}
For simplicity, we assume that $(\omega,t,x)\in\Omega\times\R\times\R^5$. Let $p>6,$ by Lemma \ref{random strichartz estimate} and $(2,\frac{10}{3})\in\Lambda_0$, we derive that
\begin{equation*}
\begin{split}
\|\langle\nabla\rangle^{s+5a(\frac{3}{10}-\frac{1}{5})}e^{it\Delta}f^\omega\|_{L^p_\omega L_t^2L_x^5}
\lesssim\sqrt p\|f\|_{H_x^s},
\end{split}
\end{equation*}
i.e.
\begin{equation*}
\begin{split}
\|\langle\nabla\rangle^{s+\frac{1}{2}a}e^{it\Delta}f^\omega\|_{L^p_\omega L_t^2L_x^5}
\lesssim\sqrt p\|f\|_{H_x^s},
\end{split}
\end{equation*}
Similarly, using Lemma \ref{random strichartz estimate}, $(4,\frac{5}{2})\in\Lambda_0$ and $(6,\frac{30}{13})\in\Lambda_0$, we find
\begin{equation*}
\begin{split}
\|\langle\nabla\rangle^{s+a} e^{it\Delta}f^\omega\|_{L^p_\omega L_t^4L_x^5}
&\lesssim\sqrt p\|f\|_{H_x^s},\\
\|\langle\nabla\rangle^{s+\frac{3}{4}a} e^{it\Delta}f^\omega\|_{L^p_\omega L_{t,x}^4}
&\lesssim\sqrt p\|f\|_{H_x^s},\\
\|\langle\nabla\rangle^{s+\frac{1}{2}a} e^{it\Delta}f^\omega\|_{L^p_\omega L_t^6L_x^3}
&\lesssim\sqrt p\|f\|_{H_x^s}.
\end{split}
\end{equation*}

Since $s>\frac{5}{4}-\frac{1}{4}a>-\frac{1}{2}a=\max\{-a,-\frac{1}{2}a,-\frac{3}{4}a\}$ for $a\in\mathbb{N}$, we have
\begin{equation*}
\|e^{it\Delta}f^\omega\|_{Y(\R)}\lesssim \sqrt p\|f\|_{H_x^s}.
\end{equation*}
Using Lemma \ref{L-D estimate2}, we get
\begin{equation}
\mathbb{P}(\{\omega\in\Omega:\|e^{it\Delta} f^\omega\|_{Y(\R)}>\lambda\})\leq C\exp\{-c\lambda^2\|f\|_{H_x^s(\R^5)}^2\}.
\end{equation}
This completes the proof of Proposition \ref{Y random bound}.
\end{proof}

Finally, we complete the proof of Theorem \ref{main theorem}.
\begin{proof}[Proof of Theorem \ref{main theorem}]
It is sufficient to prove the case of $s<0$, see Remark \ref{s<0}.
Fix $(\omega,t,x)\in\Omega\times\R\times\R^5.$
Making a high-low frequency decomposition for the random initial data
\begin{equation*}
u_0^\omega=P_{<N_0}u_0^\omega+P_{\geq N_0}u_0^\omega,
\end{equation*}
where $N_0\in 2^{\mathbb{N}}$ to be determined later. Then, let $w$ satisfy equation \eqref{NLS-hartree-w2} with the initial data $w_0(x)=P_{< N_0}u_0^\omega$, i.e.
\begin{equation}
\begin{cases} 
i\pa_t w+\Delta w=(|\cdot|^{-4}\ast|w|^2)w+e,\\
w(0,x)=P_{<N_0}u_0^\omega,
\end{cases}
\end{equation}
where
\begin{equation}
e=(|\cdot|^{-4}\ast|u|^2)u-(|\cdot|^{-4}\ast|w|^2)w,
\end{equation}
and $v$ satisfies the equation
\begin{equation}\label{section 6 v}
\begin{cases} 
i\pa_t v+\Delta v=0,\\
v(0,x)=P_{\geq N_0}u_0^\omega.
\end{cases}
\end{equation}
From \eqref{section 6 v}, we get $v(t,x)=e^{it\Delta}P_{\geq N_0}u_0^\omega$. In order to prove Theorem \ref{main theorem}, it is sufficient to show that the conditions \eqref{deterministic Condition1} and \eqref{deterministic Condition2} can be satisfied.
By Proposition \ref{Y random bound} and boundeness of the operator $P_{\geq N_0}$ (Lemma \ref{Bernstein estimates}), we get
\begin{equation}
\|v\|_{Y(\R)}=\|e^{it\Delta}P_{\geq N_0}u_0^\omega\|_{Y(\R)}\leq C\|u_0\|_{H_x^s}.
\end{equation}
Combining this with $u_0\in H_x^s$,
\begin{equation}
\|u^\omega_0\|_{H_x^s}+\|v\|_{Y(\R)}<C\|u_0\|_{H_x^s}.
\end{equation}

It remains to prove condition \eqref{deterministic Condition2}. Using Lemma \ref{L-D estimate1}, for any $p\geq 2$, we have
\begin{equation*}
\|w_0\|_{L_\omega^pL_x^2}\lesssim\sqrt p\|\Box_j P_{<N_0}u_0\|_{L_x^2\ell_j^2}\lesssim\sqrt p\|P_{<N_0}u_0\|_{L_x^2}\lesssim\sqrt p N_0^{-s}\|u_0\|_{H_x^s},
\end{equation*}
and
\begin{equation*}
\|w_0\|_{L_\omega^p\dot H^1_x}\lesssim\sqrt p N_0^{1-s}\|u_0\|_{H_x^s}.
\end{equation*}
According to H\"{o}lder's inequality, Young's inequality and the Hardy-Littlewood-Sobolev inequality \eqref{Hardy-Littlewood Sobolev inequality formula}, the potential energy can be control as follows:
\begin{equation*}
\begin{split}
\left|\iint_{\R^5\times\R^5}\frac{|u^\omega(0,x)|^2|u^\omega(0,y)|^2}{|x-y|^4}{\rm d}x{\rm d}y\right|
&=\|(|\cdot|^{-4}\ast|u^\omega_0|^2)|u^\omega_0|^2\|_{L_x^1}\\
&\lesssim
\||\cdot|^{-4}\ast|u^\omega_0|^2\|_{L_x^{\frac{5}{4}}}\||u^\omega_0|^2\|_{L_x^5}\\
&\lesssim
\||u^\omega_0|^2\|_{L_x^1}\|u^\omega_0\|^2_{L_x^{10}}\\
&\lesssim
\|u^\omega_0\|_{L_x^2}^2\|u^\omega_0\|^2_{L_x^{10}}.
\end{split}
\end{equation*}
Using Minkowski's inequality, Lemma \ref{p-q estimate} and Lemma \ref{L-D estimate1},  for any $p\geq10$, we have
\begin{equation*}
\|u^\omega_0\|_{L_\omega^p L_x^{10}}
\lesssim
\sqrt p\|\Box_ju_0\|_{L_x^{10}\ell_j^2}
\lesssim
\sqrt p\|\langle\nabla\rangle^{-2a}\Box_j u_0\|_{L_x^{2}\ell_j^2}
\lesssim
\sqrt p\|\langle\nabla\rangle^{-2a}u_0\|_{L_x^2}.
\end{equation*}
Since $s>-2a$, 
\begin{equation*}
\|u^\omega_0\|_{L_\omega^p L_x^{10}}
\lesssim
\sqrt p\|u_0\|_{H_x^s}.
\end{equation*}
And it is easy to obtain that
\begin{equation*}
\|u^\omega_0\|_{L_\omega^pL_x^2}\lesssim\sqrt p\|u_0\|_{H_x^s}.
\end{equation*}
For any $M\geq 1,$ let $\Omega_M$ be defined by
\begin{equation}\label{Omega M}
\begin{split}
\Omega_M=
&\left\{\omega\in\Omega:\|u^\omega_0\|_{H_x^s}+\|v\|_{Y(\R)}<M\|u_0\|_{H_x^s};\right.\\
&\qquad\qquad
\left.N_0^s\|w_0\|_{L_x^2}+N_0^{s-1}\|w_0\|_{\dot H^1_x}+\|u_0^\omega\|_{L_x^2}+\|u_0^\omega\|_{L_x^{10}}<M\|u_0\|_{H_x^s}\right\}.  
\end{split}
\end{equation}
Then, for any $\omega\in\Omega_M,$ we have
\begin{equation*}
M_w(0)\leq CM^2N_0^{-2s}\|u_0\|_{H_x^s}^2,   
\end{equation*}
and 
\begin{equation*}
\begin{split}
E_w(0)
&\leq
\|w_0\|_{H_x^1}^2+\|(|x|^{-4}\ast|u^\omega_0|^2)|u^\omega_0|^2\|_{L_x^1}\\
&\leq
CM^2N_0^{2-2s}\|u_0\|_{H_x^s}^2+CM^4\|u_0\|_{H_x^s}^4\\
&\leq
CM^2N_0^{2-2s}(\|u_0\|_{H_x^s}^2+M^2\|u_0\|_{H_x^s}^4).
\end{split} 
\end{equation*}
Therefore, for any $\omega\in\Omega_M,$ we can get that $u_0,$ $v$ and $w_0$ satisfy the conditions \eqref{deterministic Condition1} and \eqref{deterministic Condition2}, where 
\begin{equation*}
A=A(M,\|u_0\|_{H_x^s})
\coloneqq
\max\{CM\|u_0\|_{H_x^s},CM^2\|u_0\|_{H_x^s}^2,CM^4\|u_0\|_{H_x^s}^4\},
\end{equation*}
and 
\begin{equation*}
N_0\lesssim\left(\frac{1}{2}A\right)^{2s+\frac{a}{2}-\frac{3}{2}}.
\end{equation*}
Let $\tilde\Omega=\bigcup_{M=1}^\infty\Omega_M,$ we find that
$\mathbb{P}(\tilde\Omega)=1$. More precisely, according to Lemma \ref{L-D estimate2},
\begin{equation*}
\begin{split}
\mathbb{P}(\tilde\Omega^c)
=\mathbb{P}(\bigcap_{M=1}^\infty\Omega^c_M)
\leq\lim_{M\to\infty}\mathbb{P}(\Omega^c_M)
\overset{\text{Lemma \ref{L-D estimate2}}}{\lesssim}\lim_{M\to\infty}\exp\{-cM^2\}=0.
\end{split}
\end{equation*}
Hence $\mathbb{P}(\tilde\Omega)=1.$ Then, by using Theorem \ref{deterministic problem}, for any $\omega\in\tilde\Omega$, or almost sure $\omega\in\Omega$, we can get \eqref{main th scattering}, this completes the proof of Theorem \ref{main theorem}.
\end{proof}

\noindent

\textbf{Acknowledgements}. T. Zhao was supported by National Natural Science Foundation of China (Grant No. 12101040 and 12271051).

\begin{center}

\end{center}


\begin{thebibliography}{99}
\bibitem{Bahouri-Chemin-Danchin}
H. Bahouri, J. Y. Chemin and R. Danchin, \emph{Fourier analysis and nonlinear partial differential equations}, Springer, (2011).

\bibitem{Benyi-Oh-Pocovnicu 2015}
$\acute{\text{A}}$. B$\acute{\text{e}}$nyi, T. Oh, and O. Pocovnicu, \emph{On the probabilistic Cauchy theory of the cubic nonlinear Schr\"{o}dinger equation on $\R^d$, $d\geq 3$}, Transactions of the American Mathematical Society, Series B 2.1 (2015): 1-50.

\bibitem{Benyi-Oh-Pocovnicu 2019}
$\acute{\text{A}}$. B$\acute{\text{e}}$nyi, T. Oh, and O. Pocovnicu, \emph{Higher order expansions for the probabilistic local Cauchy theory of the cubic nonlinear Schr\"{o}dinger equation on $\R^3$}, Transactions of the American Mathematical Society, Series B 6.4 (2019): 114-160.

\bibitem{Bourgain-1994}
J. Bourgain, \emph{Periodic nonlinear Schr\"{o}dinger equation and invariant measures}, Communications in Mathematical Physics 166 (1994): 1-26.

\bibitem{Bourgain-1996}
J. Bourgain, \emph{Invariant measures for the $2d$-defocusing nonlinear Schr\"{o}dinger equation}, Communications in Mathematical physics 176.2 (1996): 421-445.

\bibitem{Bourgain 1998(267-297)}
J. Bourgain, \emph{Scattering in the energy space and below for $3D$ NLS}, Journal d'Analyse Math$\acute{\text{e}}$matique 75.1 (1998): 267-297.

\bibitem{Bringmann 2020(1011-1050)}
B. Bringmann, \emph{Almost-sure scattering for the radial energy-critical nonlinear wave equation in three dimensions}, Analysis and PDE 13.4 (2020): 1011-1050.

\bibitem{Bringmann 2021(1931-1982)}
B. Bringmann, \emph{Almost sure scattering for the energy critical nonlinear wave equation}, American Journal of Mathematics 143.6 (2021): 1931-1982.

\bibitem{Burq-Tzvetkov1}
N. Burq, N. Tzvetkov, \emph{Random data Cauchy theory for supercritical wave equations I: Local theory}, Inventiones mathematicae 173 (2008): 449-475.

\bibitem{Burq-Tzvetkov2}
N. Burq, N. Tzvetkov, \emph{Random data Cauchy theory for supercritical wave equations. II. A global existence
result}, Inventiones mathematicae 173.3 (2008): 477-496.

\bibitem{Cao-Guo}
D. Cao and Q. Guo, \emph{Divergent solutions to the $5D$ Hartree equations}, Colloquium Mathematicum, 125(2011): 255–287.

\bibitem{Cardoso-Guzman-Pastor 2022}
M. Cardoso, C. Guzman, and A. Pastor, \emph{Global well-posedness and critical norm concentration for inhomogeneous biharmonic NLS}, Monatshefte f\"{u}r Mathematik (2022): 1-29.

\bibitem{Cazenave}
T. Cazenave, \emph{Semilinear Schr\"{o}dinger equations}, Vol. 10. American Mathematical Soc., (2003).

\bibitem{Christ-Colliander-Tao}
M. Christ, J. Colliander, T. Tao, \emph{Ill-posedness for nonlinear Schr\"{o}dinger and wave equations}, arXiv preprint math/0311048 (2003).

\bibitem{Colliander-Keel-Staffilani-Takaoka-Tao}
J. Colliander, M. Keel, G. Staffilani, H. Takaoka and T. Tao. \emph{Global existence and
scattering for rough solutions of a nonlinear Schr\"{o}dinger equation on $\R^3$}, Communications on Pure and Applied Mathematics 57.8 (2004): 987-1014.

\bibitem{Dodson-Luhrmann-Mendelson 2020}
B. Dodson, J. L\"{u}hrmann, and D. Mendelson, \emph{Almost sure scattering for the $4D$ energy-critical defocusing nonlinear wave equation with radial data}. American Journal of Mathematics 142.2 (2020): 475-504.

\bibitem{Dodson-Luhrmann-Mendelson 2019}
B. Dodson, J. L\"{u}hrmann, and D. Mendelson, \emph{Almost sure local well-posedness and scattering for the $4D$ cubic nonlinear Schr\"{o}dinger equation}, Advances in Mathematics 347 (2019): 619-676.


\bibitem{Ginibre-Velo 1998(29-60)}
J. Ginibre and G. Velo, \emph{Scattering theory in the energy space for a class of Hartree equations}, Contemporary Mathematics 263 (2000): 29-60.

\bibitem{Keel-Tao}
M. Keel and T. Tao, \emph{Endpoint Strichartz estimates}, American Journal of Mathematics 120.5 (1998): 955-980.

\bibitem{Killip-Murphy-visan}
R. Killip, J. Murphy, and M. Visan, \emph{Almost sure scattering for the energy-critical NLS with radial data below $H^1(\R^4)$}, Communications in Partial Differential Equations 44.1 (2019): 51-71.

\bibitem{Killip-Visan-Zhang 2009}
R. Killip, M. Visan, and X. Zhang. \emph{The mass-critical nonlinear Schr\"{o}dinger equation with radial data in dimensions three and higher}, Analysis and PDE 1.2 (2009): 229-266.

\bibitem{Krieger-Lenzman-Raphael}
J. Krieger, E. Lenzman and P. Rapha\"{e}l, \emph{On stability of pseudo-conformal blowup for $L^2$-critical Hartree NLS}, Annales Henri Poincar$\acute{\text{e}}$, Vol. 10. No. 6. (2009): 1159–1205.

\bibitem{Lebowitz-Rose-peer}
J. Lebowitz, H. Rose, and E. Speer, \emph{Statistical mechanics of the nonlinear Schr\"{o}dinger equation}, Journal of Statistical Physics 50.3-4 (1988): 657-687.

\bibitem{Lieb-Loss}
E.Lieb and M. Loss, \emph{Analysis}, Vol. 14. American Mathematical Soc., (2001).

\bibitem{Luhrmann-Mendelson}
J. L\"{u}hrmann and D. Mendelson. \emph{Random data Cauchy theory for nonlinear wave equations of power-type on $\R^3$}, Communications in Partial Differential Equations 39.12 (2014): 2262-2283.

\bibitem{Luhrmann-Mendelson-2016}
J. L\"{u}hrmann and D. Mendelson. \emph{On the almost sure global well-posedness of energy sub-critical nonlinear wave equation $\R^3$}, New York J. Math. 22 (2016): 209-227.

\bibitem{Luo}
Y, Luo, \emph{Almost sure scattering for the defocusing cubic nonlinear Schr\"{o}dinger equation on $\mathbb {R}^3\times\mathbb{T}$}, arXiv preprint arXiv:2304.12914 (2023).

\bibitem{Miao-Xu-Zhao 2007(605-627)}
C. Miao, G. Xu and L. Zhao, \emph{Global well-posedness and scattering for the energy-critical, defocusing Hartree equation for radial data}, Journal of Functional Analysis 253.2 (2007): 605-627.

\bibitem{Miao-Xu-Zhao 2008(PDE 22-44)}
C. Miao, G. Xu and L. Zhao, \emph{The Cauchy problem of the Hartree equation}, Journal of Partial Differential Equations 21.1 (2008): 22.

\bibitem{Miao-Xu-Zhao 2009(PDE 1831-1852)}
C. Miao, G. Xu and L. Zhao, \emph{Global well-posedness and scattering for the defocusing $H^{\frac{1}{2}}$-subcritical Hartree equation in $\R^d$}, Annales de l'Institut Henri Poincar$\acute{\text{e}}$ C, Analyse non lin$\acute{\text{e}}$aire, 26(2009): 1831–1852.

\bibitem{Miao-Xu-Zhao 2009(213-236)}
C. Miao, G. Xu and L. Zhao, \emph{Global well-posedness, scattering and blow-up for the energy-critical, focusing Hartree equation in the radial case}, Colloquium Mathematicum, 114(2009): 213–236.

\bibitem{Miao-Xu-Zhao 2009(49-79)}
C. Miao, G. Xu and L. Zhao, \emph{Global well-posedness and scattering for the mass-critical Hartree equation with radial data}, Journal de Math$\acute{\text{e}}$matiques Pures et Appliqu$\acute{\text{e}}$es 91.1 (2009): 49-79.

\bibitem{Miao-Xu-Zhao 2010(23-50)}
C. Miao, G. Xu and L. Zhao, \emph{On the blow up phenomenon for the mass critical focusing Hartree equation in $\R^4$}, Colloquium Mathematicum, 119:1(2010): 23–50.

\bibitem{Miao-Xu-Zhao 2011(PDE)}
C. Miao, G. Xu and L. Zhao, \emph{Global well-posedness and scattering for the energy-critical, defocusing Hartree equation in $\R^{1+n}$}, Communications in Partial Differential Equations 36.5 (2010): 729-776.


\bibitem{Morawetz-Strauss 1972}
C. Morawetz and W. A. Strauss, \emph{Decay and scattering of solutions of a nonlinear relativistic wave equation}, Communications on Pure and Applied Mathematics 25.1 (1972): 1-31.

\bibitem{Murphy 2019}
J. Murphy, \emph{Random data final-state problem for the mass-subcritical NLS in $L^2$}, Proceedings of the American Mathematical Society 147.1 (2019): 339-350.

\bibitem{Muscalu-Schlag}
C. Muscalu and W. Schlag. \emph{Classical and Multilinear Harmonic Analysis: Volume 1}, Vol. 137. Cambridge University Press, (2013).

\bibitem{Nakanishi 1999(107-118)}
K. Nakanishi, \emph{Energy scattering for Hartree equations}, Mathematical Research Letters 6.1 (1999): 107-118.

\bibitem{Nakanishi-Yamamoto}
K. Nakanishi and T. Yamamoto, \emph{Randomized final-data problem for systems of nonlinear Schr\"{o}dinger equations and the Gross-Pitaevskii equation}, Mathematical Research Letters 26.1 (2019): 253-279.


\bibitem{Pausader-FNLS-radial}
B. Pausader. \emph{Global well-posedness for energy critical fourth-order Schr\"{o}dinger equations in the radial case}, Dynamics of Partial Differential Equations 4.3 (2007): 197-225.

\bibitem{Oh-Pocovnicu-1}
T. Oh and O. Pocovnicu, \emph{Probabilistic global well-posedness of the energy-critical defocusing quintic nonlinear wave equation on $\R^3$}, Journal de Math$\acute{\text e}$matiques Pures et Appliqu$\acute{\text e}$es 105.3 (2016): 342-366.

\bibitem{Pocovnicu1}
O. Pocovnicu, \emph{Almost sure global well-posedness for the energy-critical defocusing nonlinear wave equation on $\R^d$, $d = 4$ and $5$},  Journal of the European Mathematical Society 19.8 (2017): 2521-2575.


\bibitem{Pocovnicu-Wang}
O. Pocovnicu and Y. Wang, \emph{An $L^p$-theory for almost sure local well-posedness of the nonlinear Schr\"{o}dinger equations}, Comptes Rendus. Mathématique 356.6 (2018): 637-643.

\bibitem{Shen-Soffer-Wu 2}
J. Shen, A. Soffer, and  Y. Wu, \emph{Almost sure well-posedness and scattering of the $3D$ cubic nonlinear Schr\"{o}dinger equation}, Communications in Mathematical Physics 397.2 (2023): 547-605.

\bibitem{Shen-Soffer-Wu 1}
J. Shen, A. Soffer, and  Y. Wu, \emph{Almost sure scattering for the nonradial energy-critical NLS with arbitrary regularityin $3D$ and $4D$ cases}, arXiv preprint arXiv:2111.11935 (2021).

\bibitem{Spitz}
M, Spitz, \emph{Almost sure local well posedness and scattering for the energy-critical cubic nonlinear Schr\"{o}dinger equation with supercritical data}, Nonlinear Anal., Theory Methods Appl., Ser. A, Theory Methods 229 (2023): 33. 

\bibitem{Sun-Xia}
C. Sun and B. Xia, \emph{Probabilistic well-posedness for supercritical wave equation on $\mathbb{T}^3$}, Illinois Journal of Mathematics 60.2 (2016): 481-503.

\bibitem{Tao 2005(57-80)}
T. Tao, \emph{Global well-posedness and scattering for the higher-dimensional energy-critical nonlinear Schr\"{o}dinger equation for radial data}, New York Journal of Mathematics, 11(2005): 57-80.

\bibitem{Tao-Zhao}
L, Tao, and T, Zhao, \emph{Random data final-state problem of fourth-order inhomogeneous NLS}, Journal of Differential Equations 369 (2023): 353-382.

\bibitem{Visan}
M. Visan, \emph{The defocusing energy-critical nonlinear Schr\"{o}dinger equation in higher dimensions}, Duke Mathematical Journal 138.2 (2007): 281–374.








\end{thebibliography}
\end{document}